\documentclass{amsart}
\usepackage[latin1]{inputenc}
\usepackage[T1]{fontenc}
\usepackage{lmodern}
\usepackage[english]{babel}
\usepackage{microtype}

\usepackage{amsmath,amssymb,amsfonts,amsthm}
\usepackage{mathtools,accents}
\usepackage{mathrsfs}
\usepackage{aliascnt}
\usepackage{braket}
\usepackage{bm}

\usepackage[citecolor=blue,colorlinks]{hyperref}
\addto\extrasenglish{}

\usepackage{enumerate}
\usepackage{xcolor}

\usepackage{aliascnt}

\makeatletter
\def\newaliasedtheorem#1[#2]#3{
  \newaliascnt{#1@alt}{#2}
  \newtheorem{#1}[#1@alt]{#3}
  \expandafter\newcommand\csname #1@altname\endcsname{#3}
}
\makeatother

\theoremstyle{plain}
\newtheorem{theorem}{Theorem}[section]
\newaliasedtheorem{lemma}[theorem]{Lemma}
\newaliasedtheorem{prop}[theorem]{Proposition}
\newaliasedtheorem{claim}[theorem]{Claim}
\newaliasedtheorem{corollary}[theorem]{Corollary}

\theoremstyle{definition}
\newaliasedtheorem{definition}[theorem]{Definition}
\newaliasedtheorem{example}[theorem]{Example}

\theoremstyle{remark}
\newaliasedtheorem{remark}[theorem]{Remark}

\numberwithin{equation}{section}

\def\eps{\varepsilon}
\def\R{{\mathbb{R}}}
\def\N{{\mathbb{N}}}
\def\ra{\rightarrow}
\def\a{{\alpha}}
\def\i{{\infty}}

\title[Double Scattering Channels for NLS]{Double Scattering Channels  
for $1D$ NLS in the Energy Space and its Generalization to Higher Dimensions}

\author[L. Forcella \and N. Visciglia]{Luigi Forcella \and Nicola Visciglia}

\address{Luigi Forcella\hfill\break Scuola Normale Superiore, Piazza dei Cavalieri, 7, 56126 Pisa, Italy}
\email{luigi.forcella@sns.it}

\address{Nicola Visciglia \hfill\break Dipartimento di Matematica, Universit\`a di Pisa, Largo Bruno Pontecorvo, 5, 56100 Pisa, Italy}
\email{nicola.visciglia@unipi.it}

\begin{document}

\maketitle

\begin{abstract}
We consider a class of $1D$ NLS perturbed with a steplike potential.
We prove that the nonlinear solutions satisfy the double scattering channels
in the energy space. 
The proof is based on 
concentration-compactness/rigidity method. We prove moreover that in dimension higher than one, classical scattering holds if the potential is periodic in all but one dimension and is steplike and repulsive in the remaining one. 
\end{abstract}

\section{Introduction}\label{intro}

The main motivation of this paper is the analysis of the behavior for large times 
of solutions to the following $1D$ Cauchy problems (see below for a
suitable generalization in higher dimensions):
\begin{equation}\label{NLSV}
\left\{ \begin{aligned}
i\partial_{t}u+\partial_{x}^2 u-Vu&=|u|^{\alpha}u, \quad (t,x)\in \mathbb{R}\times 
\mathbb{R}, \quad \alpha>4\\
u(0)&=u_0\in H^1(\mathbb{R})
\end{aligned}\right.,
\end{equation}
namely we treat the $L^2$-supercritical defocusing power nonlinearities, and $V:\R\to\R$ is a real time-independent steplike potential. More precisely we assume that
$V(x)$ has two different asymptotic behaviors at $\pm\infty$:
\begin{equation}\label{differentlimit}
a_+=\lim_{x\rightarrow+ \infty} V(x)\neq \lim_{x\rightarrow -\infty} V(x)=a_-.
\end{equation}
In order to simplify the presentation we shall assume in our treatment 
\begin{equation*}
a_+=1 \quad\hbox{ and }\quad a_-=0,
\end{equation*}
but of course the arguments and the results below
can be extended to the general case $a_+\neq a_-$.
Roughly speaking the Cauchy problem \eqref{NLSV} looks like
the 
following Cauchy problems 
respectively for $x>>0$ and $x<<0$: 
\begin{equation}\label{NLS}
\left\{\begin{aligned}
i\partial_{t}v+\partial_{x}^2v&=|v|^{\alpha}v\\
v(0)&=v_0\in H^1(\mathbb{R})
\end{aligned} \right.
\end{equation}
and \begin{equation}\label{NLS1}
\left\{ \begin{aligned}
i\partial_{t}v+(\partial_{x}^2-1)v&=|v|^{\alpha}v\\
v(0)&=v_0\in H^1(\mathbb{R})
\end{aligned} \right..
\end{equation}
\noindent We recall that in $1D,$ the long time behavior of solutions to \eqref{NLS} 
(and also to \eqref{NLS1}) 
has been first obtained in the work by Nakanishi (see \cite{N}), who proved that
the solutions to \eqref{NLS} (and also \eqref{NLS1})
scatter to a free wave in $H^{1}(\mathbb{R})$
(see \autoref{def-classic} for a precise definition of scattering from nonlinear to linear solutions in a general framework).  
The Nakanishi argument is a combination of the induction
on the energy in conjunction with a suitable version of  
Morawetz inequalities with time-dependent weights. Alternative proofs 
based on the use of the interaction Morawetz estimates, first introduced in \cite{CKSTT},
have been obtained later (see 
\cite{CHVZ, CGT, PV, Vis} and the references therein). 
\newline
As far as we know, there are not results available in the literature about the
long time behavior of solutions to NLS perturbed by a steplike potential,
and this is the main motivation of this paper. 

We recall that in physics literature the steplike potentials are called
{\em barrier potentials} and are very useful to study the interactions of particles 
with the boundary of a solid (see  Gesztesy \cite{G} and Gesztesy, Noewll and P\"otz \cite{GNP} for more details). We also mention the paper
\cite{DaSi} where, in between other results,
it is studied via the twisting trick the long time behavior of solutions
to the propagator $e^{i t(\partial_x^2 - V)}$, where $V(x)$ is steplike
(see below for more details on the definition of the double scattering channels).
For a more complete list of references devoted to the analysis of steplike potentials
we refer to \cite{DS}.
Nevertheless, at the best of our knowledge, no results are available about
the long time behavior of solutions to nonlinear Cauchy problem \eqref{NLSV} with a steplike potential. 
\\
\\
It is worth mentioning that in $1D$,
we can rely on the Sobolev embedding 
$H^1(\mathbb{R} )\hookrightarrow L^\infty(\mathbb{R})$. Hence it is straightforward 
to show that the 
Cauchy problem \eqref{NLSV} is locally well posed in the energy space $H^1(\mathbb{R})$. For higher dimensions the local well posedness theory is still well known, see for example Cazenave's book \cite{CZ}, once the good dispersive properties of the linear flow are established. Moreover, thanks to the defocusing character of the nonlinearity,
we can rely on the conservation of the mass and of the energy below, valid in any dimension:
\begin{equation}\label{consmass}
\|u(t)\|_{L^2(\mathbb{R}^d)}=\|u(0)\|_{L^2(\mathbb{R}^d)},
\end{equation}
and 
\begin{equation}\label{consen}
E(u(t)):=\frac{1}{2}\int_{\mathbb{R}^d}\bigg(|\nabla u(t)|^2+V|u(t)|^2+\frac{2}{\alpha+2}|u(t)|^{\alpha+2}\bigg) dx=E(u(0)),
\end{equation}
in order to deduce that the solutions are global. Hence for any initial datum $u_0\in H^1(\mathbb{R}^d)$ there exists one 
unique global solution $u(t,x)\in {\mathcal C}(\mathbb{R}; H^1(\mathbb{R}^d))$ to 
\eqref{NLSV} for $d=1$ (and to \eqref{NLSV-d} below in higher dimension).

It is well-known that a key point in order to study the  long time behavior of nonlinear solutions
is a good knowledge of the dispersive properties
of the linear flow, namely the so called Strichartz estimates.
A lot of works have been written in the literature about the topic, both in $1D$ and in higher dimensions. We briefly mention 
\cite{AY, CK, DF, GS, W1, W2, Y} 
for the one dimensional case and \cite{BPST-Z, GVV, JK, JSS, R, RS} for the higher dimensional case, referring to the bibliographies contained in these papers for a more detailed list of works on the subject.
It is worth mentioning that in all the papers mentioned above the potential perturbation
is assumed to decay at infinity, hence steplike potential are not allowed.
Concerning contributions in the literature to NLS perturbed by a decaying potential
we have several results, in between we quote the following 
most recent ones: \cite{BV, CA, CGV, GHW, H, La, Li, LZ}, and all the references therein.

At the best of our knowledge, the unique paper where the dispersive properties of the corresponding $1D$ 
linear flow perturbed by a steplike potential $V(x)$ have been analyzed is \cite{DS},
where the $L^1-L^ \infty$ decay estimate in $1D$ is proved:
\begin{equation}\label{disp}
\|e^{it(\partial_{x}^2-V)}f\|_{L^{\infty}(\mathbb{R})}\lesssim|t|^{-1/2}\|f\|_{L^1(\mathbb{R})}, \quad \forall\, t\neq0\quad \forall\, f\in L^1(\R).
\end{equation}

We point out that beside the different spatial behavior of $V(x)$ on left and on right of the line, other assumptions must be satisfied by the potential. 

There is a huge literature devoted to those spectral properties, nevertheless 
we shall not focus on it since our main point is to show how 
to go from \eqref{disp} to the analysis of the long time 
behavior of solutions to \eqref{NLSV}. We will assume therefore as black-box the dispersive relation \eqref{disp} and for its proof, under further assumptions
on the steplike potential $V(x)$, we refer to Theorem $1.1$ in \cite{DS}. 
Our first aim is to provide a nonlinear version of the 
{\em double scattering channels } that has been established in the literature in the linear context (see \cite{DaSi}).
\begin{definition}\label{def1}
Let $u_0\in H^1(\mathbb{R})$ be given and  $u(t, x)\in \mathcal {C}(\mathbb{R};H^1(\mathbb{R}))$
be the unique global solution to \eqref{NLSV} with $V(x)$ that satisfies
\eqref{differentlimit} with $a_-=0$ and $a_+=1$. Then we say that 
$u(t,x)$ satisfies the {\em double scattering channels} provided that
$$\lim_{t\rightarrow \pm \infty} \|u(t,x) - e^{it\partial_x^2} \eta_\pm 
- e^{it(\partial_x^2-1)} \gamma_\pm \|_{H^1 (\mathbb{R})}=0,
$$
for suitable $\eta_\pm, \gamma_\pm\in H^1(\mathbb{R})$.
\end{definition}

We can now state our first result in $1D$.
\begin{theorem}\label{1Dthm} 
Assume that $V:\R\to\R$  is a bounded, nonnegative potential satisfying  \eqref{differentlimit}
with $a_-=0$ and $a_+=1,$ and  \eqref{disp}. Furthermore, suppose that: 
\begin{itemize}
\item\label{hhyp0} $|\partial_xV(x)|\overset{|x|\rightarrow\infty}\longrightarrow0$; \\
\item\label{hhyp1}$\lim_{x\rightarrow +\infty}|x|^{1+\varepsilon}|V(x)-1|=0,\,
\lim_{x\rightarrow -\infty}|x|^{1+\varepsilon}|V(x)|=0\,$ for some\, $\epsilon>0$;\\
\item\label{hhyp3} $x \cdot \partial_xV (x)\leq 0.$
\end{itemize}
Then for every $u_0\in H^1(\R)$ the corresponding unique solution 
$u(t,x)\in \mathcal {C}(\mathbb{R};H^1(\mathbb{R}^d))$ to \eqref{NLSV}
satisfies the {\em double scattering channels} (according to \autoref{def1}).
\end{theorem}

\begin{remark}\label{per1D}
It is worth mentioning that the assumption  \eqref{disp} it may look
somehow quite strong.
However we emphasize that the knowledge of the estimate
\eqref{disp} provides for free informations on the long time behavior of nonlinear solutions
for small data, but in general it is more complicated to deal with large data, as it is the case
in
\autoref{1Dthm}. For instance consider the case of $1D$ NLS perturbed 
by a periodic potential. In this situation it has been established in the literature the
validity of the dispersive estimate for the linear propagator (see \cite{Cu}) 
and also the small data nonlinear scattering (\cite{CuV}). However, at the best of our knowledge,
it is unclear how to deal with the large data scattering.
\end{remark}
The proof of \autoref{1Dthm} 
goes in two steps. The first one is to show that solutions
to \eqref{NLSV} scatter to solutions of the linear problem
(see \autoref{def-classic} for a rigorous definition of scattering in a general framework); the second one is the asymptotic description of solutions
to the linear problem associated with \eqref{NLSV} in the energy space $H^1$
(see \autoref{linscat}). 
Concerning the first step we use the technique of concentration-compactness/rigidity pioneered by Kenig and Merle (see  \cite{KM1, KM2}).
Since this argument is rather general, we shall present it 
in a more general higher dimensional setting. 
\newline
More precisely in higher dimension 
we consider the following family of NLS
\begin{equation}\label{NLSV-d}
\left\{ \begin{aligned}
i\partial_{t}u+\Delta u-Vu&=|u|^{\alpha}u, \quad (t,x)\in \mathbb{R}\times 
\mathbb{R}^d\\
u(0)&=u_0\in H^1(\mathbb{R}^d)
\end{aligned}\right.,
\end{equation}
where
\begin{equation*}
\begin{cases}
 \frac4d<\alpha<\frac{4}{d-2} &\text{if}\qquad d\geq3\\
 \frac4d<\a &\text{if}\qquad d\leq2
 \end{cases}.
\end{equation*}
The potential $V(x)$ is assumed to satisfy, uniformly in $\bar x\in\R^{d-1},$  
\begin{equation}\label{difflim2}
a_-=\lim_{x_1\rightarrow-\infty} V(x_1,\bar x)\neq \lim_{x_1\rightarrow +\infty} V(x_1, \bar x)=a_+, \quad\hbox{ where }\quad x=(x_1, \bar x).
\end{equation}

Moreover we assume $V(x)$ 
periodic w.r.t. the variables $\bar x=(x_2,\dots, x_d).$ 
Namely we assume the 
existence of $d-1$ linear independent vectors $P_2,\dots,P_d\in\R^{d-1}$ such that for any fixed 
$x_1\in\R,$ the following holds:
\begin{equation}\label{periods}
\begin{aligned}
V(x_1, \bar x)=V(x_1,\bar x +k_2P_2+\dots +k_dP_d),\\
\forall\,\bar x=(x_2,\dots,x_d)\in\R^{d-1}, 
\quad \forall\, (k_2,\dots,k_d)\in\mathbb{Z}^{d-1}. 
\end{aligned}
\end{equation}
Some comments about this choice of assumptions on $V(x)$ are given in  \autoref{assVd}.
\newline

Exactly as in $1D$ case mentioned above, 
we assume 
as a black-box the dispersive estimate  
\begin{equation}\label{disp-d}
\|e^{it(\Delta-V)}f\|_{L^{\infty}(\mathbb{R}^d)}\lesssim|t|^{-d/2}\|f\|_{L^1(\mathbb{R}^d)}, \quad \forall\, t\neq0\quad \forall\, f\in L^1(\R^d).
\end{equation} 

Next we recall the classical definition of scattering from nonlinear to linear solutions
in a general setting. We recall that by classical arguments
we have that once \eqref{disp-d} is granted, then the local (and also the global, since the equation is defocusing) existence and uniqueness of solutions 
to \eqref{NLSV-d} follows by standard arguments.

\begin{definition}\label{def-classic}
Let $u_0\in H^1(\mathbb{R}^d)$ be given and  $u(t, x)\in \mathcal {C}(\mathbb{R};H^1(\mathbb{R}^d))$
be the unique global solution to \eqref{NLSV-d}. Then we say that 
$u(t,x)$ {\em scatters to a linear solution}  provided that
$$\lim_{t\rightarrow \pm \infty} \|u(t,x) - e^{it(\Delta-V)}\psi^\pm \|_{H^1 (\mathbb{R}^d)}=0
$$
for suitable $\psi^\pm\in H^1(\mathbb{R}^d).$ \end{definition}
In the sequel we will also use the following auxiliary Cauchy problems
that roughly speaking represent the Cauchy problems 
\eqref{NLSV-d} in the regions
$x_1<<0$ and $x_1>>0$ (provide that we assume 
$a_-=0$ and $a_+=1$ in \eqref{difflim2}): 
\begin{equation}\label{NLS-d}
\left\{ \begin{aligned}
i\partial_{t}u+\Delta u&=|u|^\a u \quad (t,x)\in \mathbb{R}\times 
\mathbb{R}^d\\
u(0)&=\psi\in H^1(\mathbb{R}^d)
\end{aligned}\right.,
\end{equation}
and 
\begin{equation}\label{NLS1-d}
\left\{ \begin{aligned}
i\partial_{t}u+(\Delta -1)u&=|u|^\a u \quad (t,x)\in \mathbb{R}\times 
\mathbb{R}^d\\
u(0)&=\psi\in H^1(\mathbb{R}^d)
\end{aligned}\right..
\end{equation}
Notice that those problems are respectively the analogue 
of \eqref{NLS} and \eqref{NLS1} in higher dimensional setting.
\newline

We can now state our main result about scattering from nonlinear to 
linear solutions in general dimension $d\geq 1$.
\begin{theorem}\label{finalthm} 
Let $V\in\mathcal C^1(\mathbb{R}^d;\R)$ be a bounded, nonnegative potential which satisfies \eqref{difflim2} with $a_-=0,\,a_+=1,$ \eqref{periods} and assume moreover:
\begin{itemize}
\item\label{hyp0} $|\nabla V(x_1,\bar x)|\overset{|x_1|\to\infty}\longrightarrow0$ uniformly in $\bar x\in\R^{d-1};$\\
\item\label{hyp2} the decay estimate \eqref{disp-d} is satisfied;\\
\item\label{hyp3} $x_1 \cdot \partial_{x_1}V (x)\leq 0$ for any $x\in\R^d$.
\end{itemize}
Then for every $u_0\in H^1(\mathbb{R}^d)$ the unique corresponding global solution
$u(t,x)\in \mathcal {C}(\mathbb{R};H^1(\mathbb{R}^d))$ to \eqref{NLSV-d} scatters.
\end{theorem}

\begin{remark}\label{assVd}
Next we comment about the assumptions done 
on the potential $V(x)$ along \autoref{finalthm}. Roughly speaking we assume that the potential $V(x_1,\dots,x_d)$
is steplike and repulsive w.r.t. $x_1$ and it is periodic w.r.t. $(x_2,\dots, x_d)$.
The main motivation of this choice is that this situation is reminiscent,
according with \cite{DaSi}, 
of the higher dimensional version of the $1D$ double scattering channels
mentioned above. Moreover we highlight the fact that the 
repulsivity of the potential in one unique direction is sufficient to get scattering, despite to other situations considered in the literature
where repulsivity is assumed w.r.t. the full set of variables $(x_1,\dots,x_d)$.
Another point is that along the proof of \autoref{finalthm} 
we show how to deal with a partially periodic
potential $V(x)$, despite to the fact that, at the best of our knowledge,
the large data scattering for potentials periodic w.r.t. the full set of variables
has not been established elsewhere, either in the $1D$ case (see \autoref{per1D}).\end{remark}\begin{remark}\label{repulsively}Next we discuss about the repulsivity  assumption on $V(x)$.
As pointed out in \cite{H}, this assumption on the potential plays the same role of the convexity assumption for the obstacle problem studied by Killip, Visan and Zhang in \cite{KVZ}. The author highlights the fact that both strict convexity of the obstacle and the repulsivity of the potential prevent wave packets to refocus once they are reflected by the obstacle or by the potential. 
From a technical point of view the repulsivity assumption is done 
in order to control  the right sign in the virial identities, and
hence to conclude the rigidity part of the Kenig and Merle
argument. In this paper, since we assume repulsivity only in one
direction we use a suitable version of the Nakanishi-Morawetz time-dependent estimates in order to get the rigidity part in the Kenig and Merle road map. Of course
it is a challenging mathematical question to understand
whether or not the repulsivity assumption (partial or global)
on $V(x)$ is a necessary condition in order to get scattering.
\end{remark}When we specialize in $1D,$ we are able to complete the theory of double scattering channels in the energy space. Therefore how to concern the linear part of our work, we give the following result, that in conjunction with \autoref{finalthm}
where we fix $d=1$, provides the proof of \autoref{1Dthm}. 
\begin{theorem}\label{linscat}
Assume  that $V(x)\in\mathcal C(\mathbb{R};\R)$  satisfies the following space decay rate:
\begin{equation}\label{hyp1}
\lim_{x\rightarrow +\infty}|x|^{1+\varepsilon}|V(x)-1|=\lim_{x\rightarrow -\infty}|x|^{1+\varepsilon}|V(x)|=0\quad\text{for\, some\quad} \varepsilon>0.
\end{equation} 
Then for every $\psi\in H^1(\mathbb{R})$  we have
$$\lim_{t\rightarrow \pm \infty} \|e^{it(\partial_x^2-V)} \psi - e^{it\partial_x^2} \eta_\pm 
- e^{it(\partial_x^2-1)} \gamma_\pm \|_{H^1 (\mathbb{R})}=0$$
for suitable $\eta_\pm, \gamma_\pm\in H^1(\mathbb{R}).$
\end{theorem}
Notice that \autoref{linscat} is a purely linear statement.
The main point (compared with other results in the literature)
is that the asymptotic convergence 
is stated with respect to the $H^1$ topology and not with respect 
to the weaker $L^2$ topology.
Indeed we point out that the content of \autoref{linscat} is well-known 
and has been proved 
in \cite{DaSi} in the $L^2$ setting. However,
it seems natural to us to understand, in view of \autoref{finalthm}, whether or not the 
result can be extended in the $H^1$ setting.
In fact according with \autoref{finalthm} the asymptotic convergence of the nonlinear dynamic to linear dynamic occurs in the energy space and not only in $L^2$.  As far as we know the issue of $H^1$ linear scattering has not been previously discussed in the literature, not even in the case of a potential which decays in both directions $\pm\infty$. 

For this reason we have decided to state \autoref{linscat} as an independent result.

\subsection{Notations}\label{notations}
The spaces $L_I^{p}L^{q}=L^{p}_{t}(I;L^{q}_x(\mathbb{R}^d))$ are the usual time-space Lebesgue mixed spaces endowed with norm defined by
\begin{equation}\notag
\|u\|_{L^{p}_{t}(I;L^{q}_x(\mathbb{R}^d))}=\bigg(\int_{I}\bigg|\int_{\mathbb{R}^d}|u(t,x)|^q\,dx\bigg|^{p/q}\,dt\bigg)^{1/p}
\end{equation}
and by the context it will be clear which interval $I\subseteq\mathbb{R},$ bounded or unbounded, is considered. If $I=\mathbb{R}$ we will lighten the notation by writing $L^pL^q.$ The operator $\tau_z$ will denote the translation operator $\tau_zf(x):=f(x-z).$ If $z\in\mathbb C,$ $\Re{z}$ and $\Im{z}$ are the common notations for the real and imaginary parts of a complex number and $\bar z$ is its complex conjugate.

In what follows, when dealing with a dimension $d\geq2,$ we write $\R^d\ni x:=(x_1,\bar x)$ with $\bar x\in \R^{d-1}.$ For $x\in\R^d$ the quantity $|x|$ will denote the usual norm in $\R^d$.  

With standard notation, the Hilbert spaces $L^2(\mathbb{R}^d), H^1(\mathbb{R}^d), H^2(\mathbb{R}^d)$ will be denoted simply by $L^2, H^1, H^2$ and likely for all the Lebesgue $L^p(\mathbb{R}^d)$ spaces. By $(\cdot,\cdot)_{L^2}$ we  mean the usual $L^2$-inner product, i.e. $(f,g)_{L^2}=\int_{\mathbb{R}^d}f\bar{g}\,dx,$ $\forall\, f,g\in L^2,$ while the energy norm $\mathcal H$ is the one induced by the inner product $(f,g)_{\mathcal H}:=(f,g)_{\dot H^1}+(Vf,g)_{L^2}.$  

Finally, if $d\geq 3,$ $2^*=\frac{2d}{d-2}$ is the Sobolev conjugate of 2 ($2^*$ being $+\infty$ in dimension $d\leq2$), while if $1\leq p\leq\infty$ then $p^\prime$ is the conjugate exponent given by $p^{\prime}=\frac{p}{p-1}.$
\newline

\section{Strichartz Estimates}\label{strichartz}

The well known Strichartz estimates are a basic tool in the studying of the nonlinear Schr\"odinger equation and we will assume the validity of them in our context. 
Roughly speaking, we can say that these essential space-time estimates arise from the so-called dispersive estimate for the Schr\"odinger propagator
\begin{equation}\label{disp2}
\|e^{it(\Delta-V)}f\|_{L^{\infty}}\lesssim|t|^{-d/2}\|f\|_{L^1,} \quad \forall\, t\neq0\quad \forall\, f\in L^1,
\end{equation}
which is proved in $1D$ in \cite{DS}, under suitable assumptions
on the steplike potential $V(x)$, and we take for granted
by hypothesis. 
\\
As a first consequence we get the following Strichartz estimates
$$\|e^{it(\Delta-V)}f\|_{L^aL^b}\lesssim \|f\|_{L^2}$$
where $a,b\in [1, \infty]$ are assumed to be Strichartz admissible, namely
\begin{equation}\label{admissible}
\frac 2a=d\left(\frac 12-\frac 1 b\right).
\end{equation}

We recall, as already mentioned in the introduction,  that along our paper we are assuming the validity of the dispersive estimate \eqref{disp2} also in higher dimensional setting.

We fix from now on the following Lebesgue exponents
\begin{equation*}
r=\alpha+2,\qquad p=\frac{2\alpha(\alpha+2)}{4-(d-2)\alpha},\qquad q=\frac{2\alpha(\alpha+2)}{d\alpha^2-(d-2)\alpha-4}.
\end{equation*}
(where $\alpha$ is given by the nonlinearity in \eqref{NLSV-d}).
Next, we give the linear estimates that will be fundamental in our study:
\begin{align}\label{fxc1}
\|e^{it(\Delta-V)}f\|_{L^{\frac{4(\a+2)}{d\a}}L^r}&\lesssim \|f\|_{H^1},\\\label{fxc2}
\|e^{it(\Delta-V)}f\|_{L^\frac{2(d+2)}{d} L^\frac{2(d+2)}{d} }&\lesssim \|f\|_{H^1},\\\label{fxc3}
\|e^{it(\Delta-V)}f\|_{L^pL^r}&\lesssim \|f\|_{H^1}.
\end{align}
The last estimate that we need is (some in between) the so-called inhomogeneous Strichartz estimate for non-admissible pairs:
\begin{align}\label{str2.4}
\bigg\|\int_{0}^{t}e^{i(t-s)(\Delta-V)}g(s)\,ds \bigg\|_{L^pL^r}\lesssim\|g\|_{L^{q'}L^{r'}},
\end{align}
whose proof is contained in \cite{CW}.

\begin{remark}
In the unperturbed framework, i.e. in the absence of the potential, and for general dimensions, we refer to \cite{FXC} for comments and references about Strichartz estimates  \eqref{fxc1}, \eqref{fxc2}, \eqref{fxc3} and  \eqref{str2.4}. 
\end{remark}

\section{Perturbative nonlinear results}\label{perturbative}

The results in this section 
are quite standard and hence we skip the complete proofs which can be found for instance in \cite{BV, CZ, FXC}.
In fact the arguments involved are a compound of dispersive properties of the linear propagator
and a standard perturbation argument.

Along this section we assume that the estimate \eqref{disp-d} is satisfied by the propagator associated with the potential $V(x)$. We do not need for the moment to assume the other assumptions done on $V(x)$.

We also specify that in the sequel the couple $(p,r)$ is the one given in \autoref{strichartz}.

\begin{lemma}\label{lemma3.1}
Let $u_0\in H^1$ and assume that the corresponding solution to \eqref{NLSV-d}
satisfies $u(t,x)\in\mathcal{C}(\mathbb{R};H^{1})\cap L^{p}L^r$. 
Then $u(t,x)$ {\em scatters} to a linear solution in $H^1.$
\end{lemma}
\begin{proof}
It is a standard consequence of Strichartz estimates.
\end{proof}
\begin{lemma}\label{lemma3.2}
There exists $\varepsilon_{0}>0$ such that for any $u_0\in H^{1}$ with $\|u_0\|_{H^1}\leq\varepsilon_{0},$ the solution $u(t,x)$ to the Cauchy problem \eqref{NLSV-d} 
 {\em scatters} to a linear solution in $H^1$.
\end{lemma}
\begin{proof}
It is a simple consequence of Strichartz estimates.
\end{proof}
\begin{lemma}\label{lemma3}
For every $M>0$ there exist $\varepsilon=\varepsilon(M)>0$ and $C=C(M)>0$ such that: if $u(t,x)\in\mathcal{C}(\mathbb{R};H^1)$ is the unique global solution to \eqref{NLSV-d} and $w\in\mathcal{C}(\mathbb{R};H^1)\cap L^pL^r$ is a global solution to the perturbed problem 
\begin{equation*}
\left\{ \begin{aligned}
i\partial_{t}w+\Delta w-Vw&=|w|^{\alpha}w+e(t,x)\\
w(0,x)&=w_0\in H^1
\end{aligned} \right. 
\end{equation*}
satisfying the conditions $\|w\|_{L^pL^r}\leq M$, $\|\int_{0}^{t}e^{i(t-s)(\Delta-V)}e(s)\,ds\|_{L^pL^r}\leq \varepsilon$ and $\|e^{it(\Delta-V)}(u_0-w_0)\|_{L^pL^r}\leq\varepsilon$, then $u\in L^pL^r$ and $\|u-w\|_{L^pL^r}\leq C \varepsilon.$
\end{lemma}
\begin{proof}
The proof is contained in \cite{FXC}, see Proposition $4.7,$ and it relies on \eqref{str2.4}.
\end{proof}

\section{Profile decomposition}\label{profile}

The main content of this section is the following profile decomposition theorem. \begin{theorem}\label{profiledec} Let $V(x)\in L^\infty$ satisfies:
$V\geq 0$, \eqref{periods}, \eqref{difflim2} with $a_{-}=0$ and $a_{+}=1,$ 
the dispersive relation \eqref{disp-d} and suppose that $|\nabla V(x_1,\bar x)|\rightarrow0$ as $|x_1|\to\infty$ uniformly in $\bar x\in\R^{d-1}.$ Given a bounded sequence $\{v_n\}_{n\in\mathbb{N}}\subset H^1,$ $\forall\, J\in\mathbb{N}$ and  $\forall\,1\leq j\leq J$ there exist two sequences $\{t_n^j\}_{n\in\mathbb{N}}\subset \R,\,\{x_n^j\}_{n\in\mathbb{N}}\subset\R^d$  and $\psi^j\in H^1$ such that, up to subsequences,
\begin{equation*}
v_n=\sum_{1\leq j\leq J}e^{it_n^j(\Delta - V)}\tau_{x_n^j}\psi^j+R_n^J
\end{equation*} 
with the following properties:
\begin{itemize}
\item for any fixed $j$ we have the following dichotomy for the time parameters $t_n^j$:
\begin{align*}
either\quad t_n^j=0\quad \forall\, n\in\mathbb{N}\quad &or \quad t_n^j\overset{n\to \infty}\longrightarrow\pm\infty;
\end{align*}
\item for any fixed $j$ we have the following scenarios for the space parameters $x_n^j=(x_{n,1}^j, \bar x_n^j)\in \R\times \R^{d-1}$:
\begin{equation*}
\begin{aligned}
 either& \quad  x_n^j=0\quad \forall\, n\in\mathbb{N}\\
 or &\quad |x_{n,1}^j|\overset{n\to \infty}\longrightarrow\infty\\
or \quad x_{n,1}^j=0, \quad \bar x_{n}^j=\sum_{l=2}^dk_{n,l}^j &P_l 
\quad  with \quad k_{n,l}^j\in \mathbb Z \quad and \quad  \sum_{l=2}^d |k_{n,l}^j|
\overset{n\to \infty}\longrightarrow\infty,
\\
\quad\qquad\hbox{ where } P_l \hbox{ are given in \eqref{periods}; }
\end{aligned}
\end{equation*}
\item (orthogonality condition) for any $j\neq k$
\begin{equation*}
|x_n^j-x_n^k|+|t_n^j-t_n^k| \overset{n\ra\infty}\longrightarrow\infty;
\end{equation*}
\item (smallness of the remainder) $\forall\,\varepsilon>0\quad\exists\,J=J(\varepsilon)$ such that 
\begin{equation*}  
\limsup_{n\ra\infty}\|e^{it(\Delta - V)}R_n^J\|_{L^{p}L^{r}}\leq\varepsilon;
\end{equation*}
\item by defining $\|v\|_{\mathcal H}^2=\int (|\nabla v|^2 + V |v|^2) dx$ we have, as $n\to\infty,$
\begin{align*}\,
\|v_n\|_{L^2}^2=&\sum_{1\leq j\leq J}\|\psi^j\|_{L^2}^2+\|R_n^J\|_{L^2}^2+o(1), \quad \forall\, J\in\mathbb{N},
\\
\|v_n\|^2_{\mathcal H}=&\sum_{1\leq j\leq J}\|\tau_{x_n^j}\psi^j\|^2_{\mathcal H}+\|R_n^J\|^2_{\mathcal H}+o(1), \quad \forall\, J\in\mathbb{N};
\end{align*}
\item $\forall\, J\in\mathbb{N}$\quad and\quad $\forall\,\,2<q<2^*$ we have, as $n\to\infty,$
\begin{equation*}
\|v_n\|_{L^q}^q=\sum_{1\leq j\leq J}\|e^{it_n^j(\Delta - V)}\tau_{x_n^j}\psi^j\|_{L^q}^q+\|R_n^J\|_{L^q}^q+o(1);\\
\end{equation*}
\item with $E(v)=\frac{1}{2}\int \big(|\nabla v|^2+V|v|^2+\frac{2}{\alpha+2}|v|^{\alpha+2}\big)dx,$ we have, as $n\to\infty,$
\begin{equation}\label{energypd}
E(v_n)=\sum_{1\leq j\leq J}E(e^{it_n^j(\Delta - V)}\tau_{x_n^j}\psi^j)+E(R_n^j)+o(1),
\quad \forall\, J\in\mathbb{N}.
\end{equation}
\end{itemize}
\end{theorem}
First we prove the following lemma. 
\begin{lemma}\label{lemmapreli} 
Given a bounded sequence $\{v_n\}_{n\in\N}\subset H^1(\R^d)$ we define
\begin{equation*}
\Lambda =\left\{ w\in L^2\quad |\quad \exists \{x_k\}_{k\in\N}\quad and \quad \{n_k\}_{k\in\N}\quad\text{s. t. }\quad \tau_{x_k}v_{n_k}\overset{L^2}\rightharpoonup w\right\}
\end{equation*}
and
\begin{equation*}
\lambda=\sup\{\|w\|_{L^2},\quad w\in\Lambda\}.
\end{equation*}
Then for every $q\in(2,2^*)$ there exists a constant $M=M(\sup_n\|v_n\|_{H^1})>0$ and an exponent $e=e(d,q)>0$ such that 
\begin{equation*}
\limsup_{n\to\infty}\|v_n\|_{L^q}\leq M\lambda^e.
\end{equation*}
\end{lemma}
\begin{proof}
We consider a Fourier multiplier $\zeta$ where $\zeta$ is defined as
\begin{equation*}
C^{\infty}_c(\R^d;\R)\ni\zeta(\xi)=
\begin{cases}
1 & \text{if}\quad |\xi|\leq1\\
0 & \text{if}\quad |\xi|>2
\end{cases}.
\end{equation*} 
By setting $\zeta_R(\xi)=\zeta(\xi/R),$ we define the pseudo-differential operator with symbol $\zeta_R,$ classically given by $\zeta_R(|D|)f=\mathcal F^{-1}(\zeta_R\mathcal Ff)(x)$ and similarly we define the operator  $\tilde{\zeta}_R(|D|)$ with the associated symbol given by $\tilde{\zeta}_R(\xi)=1-\zeta_R(\xi).$ Here by $\mathcal F,\mathcal F^{-1}$ we mean the Fourier transform operator and its inverse, respectively. For any $q\in(2,2^*)$ there exists a $\epsilon\in(0,1)$ such that $H^\epsilon\hookrightarrow L^{\frac{2d}{d-2\epsilon}}=:L^q.$ Then
\begin{equation*}
\begin{aligned}
\|\tilde{\zeta}_R(|D|)v_n\|_{L^q}&\lesssim \|\langle\xi\rangle^\epsilon\tilde{\zeta}_R(\xi)\hat{v}_n\|_{L^2_\xi}\\
&= \|\langle\xi\rangle^{\epsilon-1}\langle\xi\rangle\tilde{\zeta}_R(\xi)\hat{v}_n\|_{L^2_\xi}\\
&\lesssim R^{-(1-\epsilon)}
\end{aligned}
\end{equation*}
where we have used the boundedness of $\{v_n\}_{n\in\mathbb N}$ in $H^1$ at the last step.

For the localized part we consider instead a sequence $\{y_n\}_{n\in\mathbb N}\subset\R^d$ such that 
\begin{equation*}
\|\zeta_R(|D|)v_n\|_{L^\infty}\leq2|\zeta_R(|D|)v_n(y_n)|
\end{equation*}
and we have that up to subsequences, by using the well-known properties $\mathcal F^{-1}(fg)=\mathcal F^{-1}f\ast\mathcal F^{-1}g$ and $\mathcal F^{-1}\left(f\left(\frac{\cdot}{r}\right)\right)=r^d(\mathcal F^{-1}f)(r\cdot),$ 
\begin{equation*}
\limsup_{n\to\infty}|\zeta_R(|D|)v_n(y_n)|=R^d\limsup_{n\to\infty}\left|\int\eta(Rx)v_n(x-y_n)\,dx\right|\lesssim R^{d/2}\lambda\\
\end{equation*}
where we denoted $\eta=\mathcal F^{-1}\zeta$ and we used Cauchy-Schwartz inequality. Given $\theta\in(0,1)$ such that $\frac1q=\frac{1-\theta}{2},$ by interpolation follows that
\begin{equation*}
\|\zeta_R(|D|)v_n\|_{L^q}\leq\|\zeta_R(|D|)v_n\|^{\theta}_{L^\infty}\|\zeta_R(|D|)v_n\|^{1-\theta}_{L^2}\lesssim R^{\frac{d\theta}{2}}\lambda^{\theta}
\end{equation*}
 
\begin{equation*}
\limsup_{n\to\infty}\|v_n\|_{L^q}\lesssim \left(R^{\frac{d\theta}{2}}\lambda^{\theta}+R^{-1+\epsilon}\right) 
\end{equation*}
and the proof is complete provided we select as radius $R=\lambda^{-\beta}$ with $0<\beta=\theta\left(1-\epsilon+\frac{d\theta}{2}\right)^{-1}$ and so $e=\theta(1-\epsilon)\left(1-\epsilon+\frac{d\theta}{2}\right)^{-1}.$
\end{proof}

Based on the previous lemma we can prove the following result.

\begin{lemma}
Let $\{v_n\}_{n\in \N}$ be a bounded sequence in $H^1(\R^d).$ There exists, up to subsequences, a function $\psi\in H^1$ and two sequences  $\{t_n\}_{n\in \N}\subset\R,$ $\{x_n\}_{n\in \N}\subset\R^d$ such that 
\begin{equation}\label{ex}
\tau_{-x_n}e^{it_n(\Delta-V)}v_n=\psi+W_n,
\end{equation}
where the following conditions are satisfied:
\begin{equation*}
W_n\overset{H^1}\rightharpoonup0,
\end{equation*}
\begin{equation*}
\limsup_{n\to\infty}\|e^{it(\Delta-V)}v_n\|_{L^\infty L^q}\leq C\left(\sup_n\|v_n\|_{H^1}\right)\|\psi\|_{L^2}^e
\end{equation*}
with the exponent $e>0$ given in \autoref{lemmapreli}. Furthermore, as $n\to \infty,$ $v_n$ fulfills the Pythagorean expansions below: 
\begin{equation}\label{gasp1}
\|v_n\|_{L^2}^2=\|\psi\|_{L^2}^2+\|W_n\|_{L^2}^2+o(1);
\end{equation}
\begin{equation}\label{gasp2}
\|v_n\|_{\mathcal H}^2=\|\tau_{x_n}\psi\|_{\mathcal H}^2+\|\tau_{x_n}W_n\|_{\mathcal H}^2+o(1);
\end{equation}
\begin{equation}\label{gasp3}
\|v_n\|_{L^q}^q=\|e^{it_n(\Delta-V)}\tau_{x_n}\psi\|_{L^q}^q+\|e^{it_n(\Delta-V)}\tau_{x_n}W_n\|_{L^q}^q+o(1),\qquad q\in(2,2^*).
\end{equation}
Moreover we have the following dichotomy for the time parameters $t_n$:
\begin{align}\label{parav}
either \quad t_n=0\quad \forall\, n\in\mathbb{N}\quad &or \quad t_n\overset{n\to\infty}\longrightarrow\pm\infty.
\end{align}
\item Concerning the space parameters $x_n=(x_{n,1}, \bar x_n)\in \R\times \R^{d-1}$ we have the following scenarios:
\begin{align}\label{para2v}
& either \quad  x_n=0\quad \forall\, n\in\mathbb{N}\\
\nonumber & or \quad |x_{n,1}|\overset{n\to \infty}\longrightarrow\infty\\
\nonumber or \quad x_{n,1}=0, \quad \bar x_{n}^j=\sum_{l=2}^dk_{n,l} P_l 
\quad & with \quad k_{n,l}\in \mathbb Z \quad and \quad  \sum_{l=2}^d |k_{n,l}|
\overset{n\to \infty}\longrightarrow\infty.
\end{align}
\end{lemma}
\begin{proof}
Let us choose a sequences of times $\{t_n\}_{n\in\mathbb N}$ such that 
\begin{equation}\label{time}
\|e^{it_n(\Delta-V)}v_n\|_{L^q}>\frac12\|e^{it(\Delta-V)}v_n\|_{L^\infty L^q}.
\end{equation}
According to \autoref{lemmapreli} 
we can consider a sequence of space translations such that 
\begin{equation*}
\tau_{-x_n}(e^{it_n(\Delta-V)}v_n)\overset{H^1}\rightharpoonup \psi,
\end{equation*}
which yields \eqref{ex}. Let us remark that the choice of the time sequence in \eqref{time} is possible since the norms $H^1$ and $\mathcal H$ are equivalent. 
Then 
\begin{equation*}
\limsup_{n\to\infty}\|e^{it_n(\Delta-V)}v_n\|_{L^q}\lesssim\|\psi\|_{L^2}^e,
\end{equation*}
which in turn implies by \eqref{time} that 
\begin{equation*}
\limsup_{n\to\infty}\|e^{it(\Delta-V)}v_n\|_{L^\infty L^q}\lesssim\|\psi\|_{L^2}^e,
\end{equation*}
where the exponent is the one given in \autoref{lemmapreli}. By definition of $\psi$ we can write 
\begin{equation}\label{dec2}
\tau_{-x_n}e^{it_n(\Delta-V)}v_n=\psi+W_n,\qquad W_n\overset{H^1}\rightharpoonup 0
\end{equation}
and the Hilbert structure of $L^2$ gives \eqref{gasp1}.\\

Next we prove \eqref{gasp2}. We have 
\begin{equation*}
v_n=e^{-it_n(\Delta-V)}\tau_{x_n}\psi+e^{-it_n(\Delta-V)}\tau_{x_n}W_n,\qquad W_n\overset{H^1}\rightharpoonup 0
\end{equation*}
and we conclude provided that we show
\begin{equation}\label{gasp5}(e^{-it_n(\Delta-V)}\tau_{x_n}\psi, e^{-it_n(\Delta-V)}\tau_{x_n}W_n)_{\mathcal H}\overset{n\rightarrow \infty} \longrightarrow 0.
\end{equation}
Since we have
\begin{align*}
&(e^{-it_n(\Delta-V)}\tau_{x_n}\psi, e^{-it_n(\Delta-V)}\tau_{x_n}W_n)_{\mathcal H}\\
&=(\psi, W_n)_{\dot{H}^1}+\int V(x+x_n)\psi(x)\bar{W}_n(x)\,dx
\end{align*}
and $W_n\overset{H^1}\rightharpoonup 0
$,
it is sufficient to show that 
\begin{equation}\label{gasp4}
\int V(x+x_n)\psi(x)\bar{W}_n(x)\,dx
\overset{n\rightarrow \infty} \longrightarrow 0.
\end{equation}
If (up to subsequence) $x_n\overset{n\to \infty}\longrightarrow x^*\in\R^d$ or  $|x_{n,1}|\overset{n\to \infty}\longrightarrow\infty$,
where we have splitted $x_n=(x_{n,1}, \bar x_n)\in \R\times \R^{d-1}$,  
then we have that the sequence $\tau_{-x_n}V (x)=V(x+x_n)$ pointwise converges to the function $\tilde{V}(x)\in L^{\infty}$ defined by 
\begin{equation*}
\tilde{V}(x)=
\left\{ \begin{array}{ll}
1\quad &if\quad x_{n,1}\overset{n\ra \infty}\longrightarrow+\infty\\
V(x+x^*)\quad &if\quad x_n\overset{n\ra \infty}\longrightarrow x^*\in\R^d\\
0\quad &if\quad x_{n,1}\overset{n\ra \infty}\longrightarrow-\infty
\end{array} \right.
\end{equation*}
and hence
\begin{equation*}
\begin{aligned}
\int V(x+x_n)\psi(x)\bar{W}_n(x)\,dx&=\int[V(x+x_n)-\tilde{V}(x)]\psi(x)\bar{W}_n(x)\,dx\\
&\quad+\int \tilde{V}(x)\psi(x)\bar{W}_n(x)\,dx.
\end{aligned}
\end{equation*}
The function $\tilde{V}(x)\psi(x)$ belongs to $L^2$ since $\tilde{V}$ is bounded and $\psi\in H^1$, and since $W_n\rightharpoonup0$ in $H^1$ (and then in $L^2$) we have that 
\begin{equation*}
\int \tilde{V}(x)\psi(x)\bar{W}_n(x)\,dx\overset{n\ra \infty}\longrightarrow0.
\end{equation*}
Moreover by using Cauchy-Schwartz inequality
\begin{align*}
\bigg|\int[V(x+x_n)-\tilde{V}(x)]\psi(x)\bar{W}_n(x)\,dx\bigg|\leq&\sup_{n}\|W_n\|_{L^2}\|[V(\cdot+x_n)-\tilde{V}(\cdot)]\psi(\cdot)\|_{L^2};
\end{align*}
since $\left|[V(\cdot+x_n)-\tilde{V}(\cdot)]\psi(\cdot) \right|^2\lesssim|\psi(\cdot)|^2\in L^1$ we claim, by dominated convergence theorem, that also 
\begin{equation*}
\int[V(x+x_n)-\tilde{V}(x)]\psi(x)\bar{W}_n(x)\,dx\overset{n\ra \infty}\longrightarrow0,
\end{equation*}
and we conclude \eqref{gasp4} and hence \eqref{gasp5}.
It remains to prove \eqref{gasp5} in the case when, up to subsequences, $x_{n,1}
\overset{n\rightarrow \infty}
\longrightarrow x_1^*$ and $|\bar x_n|\overset {n\rightarrow \infty}
\longrightarrow \infty$. Up to subsequences we can assume therefore that $\bar x_{n}= \bar x^*+\sum_{l=2}^d k_{n, l}P_l+o(1)$ 
with $\bar x^*\in \R^{d-1}$, $k_{n, l}\in \mathbb Z$ and
$\sum_{l=2}^d |k_{n,l}|\overset{n\rightarrow \infty} \longrightarrow \infty.$ Then by using the 
periodicity of the potential $V$ w.r.t. the $(x_2,\dots, x_d)$ variables we get:
\begin{equation*}
\begin{aligned}
&(e^{-it_n(\Delta-V)}\tau_{x_n}\psi, e^{-it_n(\Delta-V)}\tau_{x_n}W_n)_{\mathcal H}=\\
&(e^{-it_n(\Delta-V)}\tau_{(x_1^*,\bar x_n)}\psi, e^{-it_n(\Delta-V)}\tau_{(x_1^*,\bar x_n)}W_n)_{\mathcal H}+o(1)=\\
&(\tau_{(x_1^*,\bar x^*)}\psi, \tau_{(x_1^*,\bar x^*)}W_n)_{\mathcal H}+o(1)=\\
&(\psi,W_n)_{\dot H^1}+\int V(x+(x_1^*,\bar x^*))\psi(x)\bar{W}_n\,dx=o(1)
\end{aligned}
\end{equation*}
where we have used the fact that $W_n\overset{ H^1} \rightharpoonup0$. 
\newline

We now turn our attention to the orthogonality of the non quadratic term of the energy,
namely \eqref{gasp3}. The proof is almost the same of the one carried out in \cite{BV}, with some modification. \\

\noindent \emph{Case 1.} Suppose $|t_n|\overset{n\to \infty}\longrightarrow\infty.$ By \eqref{disp2} we have  $\|e^{it(\Delta-V)}\|_{L^1\ra L^{\infty}}\lesssim|t|^{-d/2}$ for any $t\neq0.$ We recall that for the evolution operator $e^{it(\Delta-V)}$ the $L^2$ norm is conserved, so  the estimate  $\|e^{it(\Delta-V)}\|_{L^{p\prime}\ra L^{p}}\lesssim|t|^{-d\left(\frac{1}{2}-\frac{1}{p}\right)}$ holds from Riesz-Thorin interpolation theorem, thus we have the conclusion
provided that $\psi\in L^1\cap L^2$. If 
this  is not the case we can conclude by a straightforward approximation argument. This implies that if $|t_n|\to\infty$ as $n\to\infty$ then for any $p\in(2,2^*)$ and for any $\psi\in H^1$
\begin{equation*}
\|e^{it_n(\Delta-V)}\tau_{x_n}\psi\|_{L^p}\overset{n\ra \infty}\longrightarrow 0.
\end{equation*}
Thus we conclude  by \eqref{dec2}. \\

\noindent\emph{Case 2.}
Assume now that  $t_n\overset{n\to \infty}\longrightarrow t^*\in\R$ and $x_n
\overset{n\to \infty}\longrightarrow x^*\in\R^d.$ In this case the proof relies on a combination of the Rellich-Kondrachov theorem 
and the Brezis-Lieb Lemma contained in \cite{BL}, provided that 
\begin{equation}\notag
\|e^{it_n(\Delta-V)}(\tau_{x_n}\psi)-e^{it^*(\Delta-V)}(\tau_{x^*}\psi)\|_{H^1}\overset{n\ra\infty}\longrightarrow0,\qquad\forall\,\psi\in H^1.
\end{equation} 
But this is a straightforward consequence of the continuity of the linear propagator (see \cite{BV} for more details).
\newline

\noindent \emph{Case 3.} It remains to consider $t_n\overset{n\to \infty}\longrightarrow t^*\in\R$ and $|x_n|\overset{n\to \infty}\longrightarrow\infty.$ 
Also here we can proceed as in \cite{BV} provided that for any $\psi\in H^1$ there exists a $\psi^*\in H^1$ such that
\begin{equation}\notag
\|\tau_{-x_n}(e^{it_n(\Delta-V)}(\tau_{x_n}\psi))-\psi^*\|_{H^1}\overset{n\ra\infty}\longrightarrow0.
\end{equation} 
Since translations are isometries in $H^1,$ it suffices to show that  for some $\psi^*\in H^1$ 
\begin{equation*}
\|e^{it_n(\Delta-V)}\tau_{x_{n}}\psi-\tau_{x_n}\psi^*\|_{H^1}
\overset{n\ra \infty}\longrightarrow0.
\end{equation*}
We decompose $x_n=(x_{n,1}, \bar x_n)\in \R\times \R^{d-1}$
and
we consider the two scenarios: $|x_{n,1}|\overset{n\ra \infty}\longrightarrow \infty$ and $
\sup_n |x_{n,1}|<\infty$.
\newline
If $x_{n,1}
\overset{n\ra \infty}\longrightarrow -\infty,$ by continuity in $H^1$ of the flow, it is enough to prove that 
\begin{equation*}
\| e^{it^*(\Delta-V)}\tau_{x_{n}}\psi-e^{it^*\Delta}\tau_{x_{n}}\psi\|_{H^{1}}
\overset{n\ra \infty}\longrightarrow 0.
\end{equation*}

We observe that 
\begin{equation*}
e^{it^*(\Delta-V)}\tau_{x_{n}}\psi-e^{it^*\Delta}\tau_{x_{n}}\psi=\int_{0}^{t^*}e^{i(t^*-s)(\Delta-V)}(Ve^{-is \Delta}\tau_{x_{n}}\psi)(s)\,ds
\end{equation*}
and hence, 
\begin{equation*}
\| e^{it^*(\Delta-V)}\tau_{x_{n}}\psi-e^{it^*\Delta}\tau_{x_{n}}\psi\|_{H^1}\leq \int_0^{t^*}
\|(\tau_{-x_n}V)e^{is\Delta}\psi\|_{H^1} ds.
\end{equation*}
We will show that 
\begin{equation}\label{s}\int_0^{t^*}
\|(\tau_{-x_n}V)e^{is\Delta}\psi\|_{H^1} ds\overset{n\ra \infty}\longrightarrow 0.
\end{equation}
Since we are assuming $x_{n,1}
\overset{n\ra \infty}\longrightarrow -\infty,$ for fixed $x\in\mathbb{R}^d$ we get $V(x+x_n)\overset{n\ra \infty}\longrightarrow 0,$  namely $(\tau_{-x_n}V)(x)\overset{n\ra \infty}\longrightarrow 0$ pointwise; since $V\in L^{\infty},$ $|\tau_{-x_n}V|^2|e^{it \Delta}\psi|^2\leq\|V\|^2_{L^{\infty}}|e^{it\Delta}\psi|^2$ and $|e^{it \Delta}\psi|^2\in L^1,$ the dominated convergence theorem yields to
\begin{equation*}
\|(\tau_{-x_n}V)e^{it \Delta}\psi\|_{L^2}\overset{n\ra \infty}\longrightarrow 0.
\end{equation*} 
Analogously, since $|x_{n,1}|
\overset{n\ra \infty}\longrightarrow\infty$
implies $| \nabla \tau_{-x_n} V(x)|
\overset{n\ra \infty}\longrightarrow  0,$ 
we obtain 
\begin{equation*}
\|\nabla(\tau_{-x_n}Ve ^{it\Delta} \psi)\|_{L^2}\leq\| 
(e^{it \Delta} \psi) \nabla 
\tau_{-x_n} V \|_{L^2}+\|
(\tau_{-x_n}V) \nabla(e^{it \Delta}\psi)\|_{L^2}\overset{n\ra \infty}\longrightarrow 
0.
\end{equation*}

We conclude 
\eqref{s} by using 
the dominated convergence theorem w.r.t the measure 
$ds$.

For the case  $x_{n,1}\overset{n\ra \infty}\longrightarrow \infty$ we proceed similarly.
 
If $\sup_{n\in \N} |x_{n,1}|<\infty,$ then
up to subsequence
$x_{n,1} \overset{n\rightarrow \infty} \longrightarrow x_1^*\in \R$.
The thesis follows by choosing 
 $\psi^*=e^{it^*(\Delta-V)}\tau_{(x_1^*,\bar x^*)}\psi,$ with $\bar x^*
 \in \R^{d-1}$ defined as follows
(see above the proof of \eqref{gasp2}):
$\bar x_{n}= \bar x^*+\sum_{l=2}^d k_{n, l}P_l+o(1)$ with
$k_{n, l}\in \mathbb Z$ and
$\sum_{l=2}^d |k_{n,l}|\overset{n\rightarrow \infty} \rightarrow \infty.$ 
\newline
Finally, it is straightforward from \cite{BV} that the conditions on the parameters \eqref{parav} and \eqref{para2v} hold. 

\end{proof}

\begin{proof}[Proof of \autoref{profiledec}] The proof of the profile decomposition theorem can be carried out as in \cite{BV} iterating the previous lemma. 
\end{proof}
\section{Nonlinear profiles}\label{nonlinearpro}

The results of this section will be crucial along the construction of the minimal element.
We recall that the couple $(p,r)$ is the one given in \autoref{strichartz}.
Moreover for every sequence $x_n\in \R^d$ we use the notation
$x_n=(x_{n,1}, \bar x_n)\in \R\times \R^{d-1}$.

\begin{lemma}\label{lem5.1}
Let $\psi\in H^1$ and $\{x_n\}_{n\in \mathbb{N}}\subset\mathbb{R}^d$ be such that 
$|x_{n,1}| \overset{n\ra \infty}\longrightarrow \infty$. Up to subsequences  we have the following estimates:
\begin{equation}\label{eq5.1}
x_{n,1} \overset{n\ra \infty}\longrightarrow -\infty
\implies \|e^{it\Delta}\psi_{n}-e^{it(\Delta-V)}\psi_{n}\|_{L^pL^r}\overset{n\ra\infty}\longrightarrow 0,
\end{equation}
\begin{equation}\label{eq5.2}
x_{n,1}
\overset{n\ra \infty}\longrightarrow +\infty \implies \|e^{it(\Delta-1)}\psi_{n}-e^{it(\Delta-V)}\psi_{n}\|_{L^pL^r}\overset{n\ra\infty}\longrightarrow 0,
\end{equation}
where $\psi_n:=\tau_{x_n}\psi.$
\end{lemma}

\begin{proof}
Assume 
$x_{n,1}
\overset{n\ra \infty}\longrightarrow -\infty$
(the case $x_{n,1}
\overset{n\ra \infty}\longrightarrow +\infty$
can be treated similarly).
We first prove that
\begin{equation}\label{eq5.3}
\sup_{n\in\mathbb{N}}\| e^{it(\Delta-V)}\psi_{n}\|_{L^{p}_{(T,\infty)}L^{r}}\overset{T\ra\infty}\longrightarrow0.
\end{equation}
Let $\varepsilon>0$. By density there exists $\tilde{\psi}\in C^{\infty}_c$ such that $\|\tilde{\psi}-\psi\|_{H^{1}}\leq\varepsilon,$ then by the estimate \eqref{fxc3}
\begin{equation*}
\|e^{it(\Delta-V)}(\tilde{\psi}_{n}-\psi_{n})\|_{L^{p}L^{r}}\lesssim\|\tilde{\psi}_{n}-\psi_{n}\|_{H^{1}}=\|\tilde{\psi}-\psi\|_{H^{1}}\lesssim\varepsilon.
\end{equation*}
Since $\tilde{\psi}\in L^{r'}$, by interpolation between 
the dispersive estimate \eqref{disp2}
and the conservation of the mass along the linear flow,
we have
\begin{equation*}
\| e^{it(\Delta-V)}\tilde{\psi}_{n}\|_{L^{r}}\lesssim|t|^{-d\left(\frac{1}{2}-\frac{1}{r}\right)}\|\tilde{\psi}\|_{L^{r'}},
\end{equation*}
and since $f(t)=|t|^{-d\left(\frac{1}{2}-\frac{1}{r}\right)}\in L^p(|t|>1),$ there exists $T>0$ such that 
\begin{equation*}
\sup_n\| e^{it(\Delta-V)}\tilde{\psi}_{n}\|_{L^{p}_{|t|\geq T}L^{r}}\leq\varepsilon,
\end{equation*}
hence we get \eqref{eq5.3}.
In order to obtain \eqref{eq5.1}, we are reduced to show that for a fixed $T>0$
\begin{equation*}
\| e^{it \Delta}\psi_{n}-e^{it(\Delta-V)}\psi_{n}\|_{L^{p}_{(0,T)}L^{r}}\overset{n\ra\infty}\longrightarrow0.
\end{equation*}
Since $w_n=e^{it \Delta}\psi_{n}-e^{it(\Delta-V)}\psi_{n}$ is the solution of the following linear Schr\"odinger equation 
\begin{equation*}
\left\{\begin{aligned}
i\partial_{t}w_n+ \Delta w_n-Vw_n&=-Ve^{it\Delta}\psi_{n}\\
w_n(0)&=0
\end{aligned}\right.,
\end{equation*}
by combining \eqref{fxc3} with the Duhamel formula we get
\begin{align*}
\|e^{it \Delta}\psi_{n}-e^{it(\Delta-V)}\psi_{n}\|_{L^{p}_{(0,T)}L^{r}}&\lesssim 
\|(\tau_{-x_n}V)e^{it \Delta}\psi\|_{L^1_{(0,T)}H^1}.
\end{align*}
The thesis follows from the dominated convergence theorem.
\end{proof}

\begin{lemma}\label{lem5.2}
Let $\{x_{n}\}_{n\in\mathbb{N}}\subset\mathbb{R}^d$ be a sequence such that 
$x_{n,1}
\overset{n\ra \infty}\longrightarrow -\infty,$
(resp. $x_{n,1}
\overset{n\ra \infty}\longrightarrow +\infty$)
and $v\in \mathcal{C}(\mathbb{R};H^{1})$ 
be the unique solution to \eqref{NLS-d} (resp. \eqref{NLS1-d}). Define  $v_{n}(t,x):=v(t,x-x_{n})$. Then, up to a subsequence, the followings hold: 

\begin{equation}\label{eq5.11}
\bigg\|\int_{0}^{t}[e^{i(t-s)\Delta}\left (|v_{n}|^{\alpha}v_{n}
\right )-e^{i(t-s)(\Delta-V)}
\left (|v_{n}|^{\alpha}v_{n}\right )]ds \bigg\|_{L^{p}L^{r}}\overset{n\ra\infty}\longrightarrow0;
\end{equation}

\begin{equation}\label{eq5.12}
\left( \hbox{resp.\,} \bigg\|\int_{0}^{t}[e^{i(t-s)(\Delta-1)}\left (|v_{n}|^{\alpha}v_{n}\right )
-e^{i(t-s)(\Delta-V)}
\left (|v_{n}|^{\alpha}v_{n}\right )]ds\bigg\|_{L^{p}L^{r}}\overset{n\ra\infty}\longrightarrow0\right).
\end{equation}
\end{lemma}
\begin{proof} Assume $x_{n,1}
\overset{n\ra \infty}\longrightarrow -\infty$
 (the case $x_{n,1}
\overset{n\ra \infty}\longrightarrow +\infty$
can be treated similarly).
Our proof starts with the observation that
\begin{equation}\label{fiseca}
\lim_{T\rightarrow \infty} 
\bigg(\sup_n \bigg\|\int_{0}^{t}e^{i(t-s)(\Delta-V)} \left (|v_{n}|^{\alpha}v_{n}\right )ds\bigg\|_{L^{p}_{(T,\infty)}L^{r}}\bigg)=0.
\end{equation}
By Minkowski inequality and the interpolation of the dispersive estimate \eqref{disp2} 
with the conservation of the mass, we have
\begin{align}\notag
\bigg\|\int_{0}^{t}e^{i(t-s)(\Delta-V)}\left(|v_{n}|^{\alpha}v_{n}\right)ds\bigg\|_{L^r_x}&\lesssim\int_{0}^{t}|t-s|^{-d\left(\frac{1}{2}-\frac{1}{r}\right)}\||v_n|^{\alpha}v_n\|_{L^{r'}_x}ds\\\notag
&\lesssim\int_{\mathbb{R}}|t-s|^{-d\left(\frac{1}{2}-\frac{1}{r}\right)}\||v|^{\alpha}v\|_{L^{r'}_x} ds= 
|t|^{-d\left(\frac 12 - \frac 1r\right)}\ast g
\end{align}
with $g(s)=\||v|^{\alpha}v(s)\|_{L^{r'}_x}.$ We conclude \eqref{fiseca} provided that we show
$|t|^{-d\left(\frac 12 - \frac 1r\right)}\ast g(t)\in L^p_t$.
By using the Hardy-Littlewood-Sobolev inequality (see for instance Stein's book \cite{ST}, p. 119) we assert
\begin{equation*}
\big\||t|^{-1+\frac{(2-d)\a+4}{2(\a+2)}}\ast g(t) 
\big\|_{L^p_t} \lesssim\||v|^{\alpha}v\|_{L^{\frac{2\alpha(\alpha+2)}{\left((2-d)\a+4\right)(\a+1)}}L^{r'}}=\|v\|^{\a+1}_{L^pL^r}.
\end{equation*}
Since $v$ scatters, then it belongs to $L^pL^r,$ and so we can deduce the validity of \eqref{fiseca}.
\newline

Consider now $T$ fixed: we are reduced to show that 
\begin{equation*}
\bigg\|\int_{0}^{t}[e^{i(t-s)\Delta}\left(|v_{n}|^{\alpha}v_{n}\right)-e^{i(t-s)(\Delta-V)}\left(|v_{n}|^{\alpha}v_{n}\right)]ds\bigg\|_{L^{p}_{(0,T)} L ^{r}}\overset{n\ra\infty}\longrightarrow0.
\end{equation*} 
As usual we observe that 
\begin{equation*}
w_n(t,x)=\int_{0}^{t}e^{i(t-s) \Delta}\left(|v_{n}|^{\alpha}v_{n}\right) ds - \int_{0}^{t}e^{i(t-s)(\Delta-V)}\left(|v_{n}|^{\alpha}v_{n}\right) ds
\end{equation*}
is the solution of the following linear Schr\"odinger equation
\begin{equation*}
\left\{\begin{aligned}
i\partial_{t}w_n+ \Delta w_n -V w_n&=-V\int_{0}^{t}e^{i(t-s) \Delta}\left(|v_{n}|^{\alpha}v_{n}\right)ds\\
w_n(0)&=0
\end{aligned}\right.,
\end{equation*}
and likely for \autoref{lem5.1} we estimate 
\begin{align*}\notag
\bigg\| \int_{0}^{t}e^{i(t-s) \Delta}&\left(|v_{n}|^{\alpha}v_{n}\right)ds-\int_{0}^{t}e^{i(t-s)(\Delta-V)}\left(|v_{n}|^{\alpha}v_{n}\right)ds \bigg\|_{L^{p}_{(0,T)}L^{r}}\\
&\lesssim \|(\tau_{-x_n}V)|v|^{\alpha}v\|_{L^1_{(0,T)}H^1}.
\end{align*}
By using the dominated convergence theorem we conclude the proof.

\end{proof}
The previous results imply the following useful corollaries.
\begin{corollary}\label{cor5.3}
Let $\{x_n\}_{n\in\mathbb{N}}\subset\mathbb{R}^d$ be a sequence such that 
$x_{n,1}
\overset{n\ra \infty}\longrightarrow -\infty,$
and let $v\in\mathcal{C}(\mathbb{R};H^1)$ be the unique solution to \eqref{NLS-d} with initial datum $v_0\in H^1$. Then
\begin{equation*}
v_n(t,x)=e^{it(\Delta-V)}v_{0,n} -i\int_{0}^{t}e^{i(t-s)(\Delta-V)}\left(|v_{n}|^{\alpha}v_{n}\right)ds+e_{n}(t,x)
\end{equation*}
where $v_{0,n}(x):=\tau_{x_n}v_0(x),$ $v_{n}(t,x):=v(t,x-x_n)$ and $\|e_n\|_{L^pL^r}
\overset{n\ra\infty}\longrightarrow 0$.
\end{corollary}
\begin{proof}
It is a consequence of \eqref{eq5.1} and \eqref{eq5.11}.
\end{proof}
\begin{corollary}\label{cor5.4} 
Let $\{x_n\}_{n\in\mathbb{N}}\subset\mathbb{R}^d$ be a sequence such that 
$x_{n,1}
\overset{n\ra \infty}\longrightarrow +\infty,$
and let $v\in\mathcal{C}(\mathbb{R};H^1)$ be the
unique solution to \eqref{NLS1-d} with initial datum $v_0\in H^1$.
Then
\begin{equation*}
v_n(t,x)=e^{it(\Delta-V)} v_{0,n}- i\int_{0}^{t}e^{i(t-s)(\Delta-V)}\left(|v_{n}|^{\alpha}v_{n}\right)ds+e_{n}(t,x)
\end{equation*}
where $v_{0,n}(x):=\tau_{x_n}v_0(x),$ $v_{n}(t,x):=v (t,x-x_n)$ and $\|e_n\|_{L^pL^r}
\overset{n\ra\infty}\longrightarrow 0$.
\end{corollary}
\begin{proof}
It is a consequence of \eqref{eq5.2} and \eqref{eq5.12}.
\end{proof}

\begin{lemma}\label{lem5.5}
Let $v(t,x)\in \mathcal C(\mathbb{R}; H^1)$  be a solution to \eqref{NLS-d} 
(resp. \eqref{NLS1-d})
and let $\psi_\pm \in H^{1}$  (resp.  $\varphi_\pm \in H^{1}$) be such that
\begin{equation*}
\begin{aligned}
&\|v(t,x)-e^{it \Delta}\psi_\pm \|_{H^1}\overset{t\rightarrow\pm\infty}
{\longrightarrow}0
\\
\bigg(\hbox{ resp. } 
&\|v(t,x)-e^{it(\Delta-1)}\varphi_\pm\|_{H^1}\overset{t\rightarrow\pm\infty}
{\longrightarrow}0\bigg).
\end{aligned}
\end{equation*}
Let $\{x_{n}\}_{n\in\mathbb{N}}\subset\R^d, \,\{t_{n}\}_{n\in\mathbb{N}}\subset\mathbb{R}$ be two sequences such that $x_{n,1}
\overset{n\ra \infty}\longrightarrow -\infty$
(resp. $x_{n,1}
\overset{n\ra \infty}\longrightarrow +\infty$) and $|t_{n}|\overset{n\ra\infty}\longrightarrow \infty.$ Let us define moreover $v_{n}(t,x):=v(t-t_{n},x-x_{n})$ and $\psi_n^\pm (x):=\tau_{x_n}\psi_\pm(x)$ 
(resp.  $\varphi_n^\pm (x)=\tau_{x_n}\varphi_\pm(x)$). 
Then, up to subsequence, we get 

\begin{align}\notag
t_n\rightarrow \pm \infty \Longrightarrow
\| e^{i(t-t_{n})\Delta}\psi_{n}^\pm -e^{i(t-t_{n})(\Delta-V)}\psi_{n}^\pm\|_{L^{p}L^{r}}\overset{n\ra\infty}\longrightarrow0\quad\emph{and}
\\\label{eq517}
\bigg\|\int_{0}^{t}[e^{i(t-s) \Delta}\big(|v_{n}|^{\alpha}
v_{n}\big)ds- e^{i(t-s)(\Delta-V)}\big(|v_{n}|^{\alpha}v_{n}\big)]ds \bigg\|_{L^{p}L^{r}}\overset{n\ra\infty}\longrightarrow0
\end{align}
\begin{align*}
\bigg(\hbox{ resp. } 
t_n\rightarrow \pm \infty \Longrightarrow
\| e^{i(t-t_{n})(\Delta-1)}\varphi_{n}^\pm -e^{i(t-t_{n})(\Delta-V)}\varphi_{n}^\pm\|_{L^{p}L^{r}}\overset{n\ra\infty}\longrightarrow0\quad\emph{and}
\\\nonumber
\bigg\|\int_{0}^{t}[e^{i(t-s)(\Delta-1)}\big(|v_{n}|^{\alpha}
v_{n}\big)ds- e^{i(t-s)(\Delta-V)}\big(|v_{n}|^{\alpha} v_{n}\big)]ds \bigg\|_{L^{p}L^{r}}\overset{n\ra\infty}\longrightarrow0 \bigg).
\end{align*}

\end{lemma}
\begin{proof} It is a multidimensional suitable version of 
Proposition 3.6 in \cite{BV}. 
Nevertheless, since in \cite{BV} the details of the proof are not given, we expose below the proof
of the most delicate estimate, namely the second estimate
in \eqref{eq517}. After a change of variable in time, proving \eqref{eq517} is clearly equivalent to prove 
\begin{equation*}
\bigg\|\int_{-t_n}^{t}e^{i(t-s) \Delta}\tau_{x_n}(|v|^{\alpha} v)(s)ds- \int_{-t_n}^te^{i(t-s)(\Delta-V)}\tau_{x_n}(|v|^{\alpha} v)(s)ds \bigg\|_{L^{p}L^{r}}\overset{n\ra\infty}\longrightarrow0.
\end{equation*}
\newline
We can focus on the case $t_n\ra \infty$ and $x_{n,1}
\overset{n\ra \infty}\longrightarrow+ \infty,$ being the other cases similar. 
\\
The idea of the proof is to split the estimate above in three different regions, i.e. $(-\infty,-T)\times\mathbb{R}^d, (-T,T)\times\mathbb{R}^d, (T,\infty)\times\mathbb{R}^d$ for some fixed $T$ which will be chosen in an appropriate way below. The strategy is to use translation in the space variable to gain smallness in the strip  $(-T,T)\times\mathbb{R}^d$ while we use smallness of Strichartz estimate in half spaces $(-T,T)^c\times\mathbb{R}^d$. Actually in $(T,\infty)$ the situation is more delicate and we will also use the dispersive relation. \\

Let us define $g(t)=\|v(t)\|^{\a+1}_{L^{(\a+1)r'}}$ and for fixed $\varepsilon>0$ let us consider $T=T(\varepsilon)>0$ such that:
\begin{equation}\label{smallness}
\begin{aligned}
\left\{ \begin{array}{ll}
\||v|^{\alpha}v\|_{L^{q'}_{(-\infty,-T)}L^{r'}}<\varepsilon\\
\||v|^{\alpha}v\|_{L^{q'}_{(T,+\infty)}L^{r'}}<\varepsilon\\
\||v|^{\alpha}v\|_{L^1_{(-\infty,-T)}H^1
}<\varepsilon\\
\left\||t|^{-d\left(\frac 12 - \frac 1r\right)}\ast g(t)\right \|_{L^p_{(T,+\infty)}}<\varepsilon
\end{array} \right..
\end{aligned}
\end{equation}
The existence of such a $T$ is guaranteed by the integrability properties of $v$ and its decay at infinity (in time). We can assume without loss of generality that $|t_n|>T.$\\
We split the term to be estimated as follows:
\begin{equation*}
\begin{split}
\int_{-t_n}^{t}e^{i(t-s)\Delta}\tau_{x_n}(|v|^{\alpha} v)(s)ds- \int_{-t_n}^te^{i(t-s)(\Delta-V)}\tau_{x_n}(|v|^{\alpha} v)(s)ds\\
=e^{it\Delta}\int_{-t_n}^{-T}e^{-is\Delta}\tau_{x_n}(|v|^{\alpha} v)(s)ds- e^{it(\Delta-V)}\int_{-t_n}^{-T}e^{-is(\Delta-V)}\tau_{x_n}(|v|^{\alpha} v)(s)ds\\
+\int_{-T}^{t}e^{i(t-s)\Delta}\tau_{x_n}(|v|^{\alpha} v)(s)ds- \int_{-T}^te^{i(t-s)(\Delta-V)}\tau_{x_n}(|v|^{\alpha} v)(s)ds.
\end{split}
\end{equation*}
By Strichartz estimate \eqref{fxc3} and the third one of \eqref{smallness}, we have, uniformly in $n,$
\begin{equation*}
\begin{aligned}
\bigg\|e^{it\Delta}\int_{-t_n}^{-T}e^{-is\Delta}\tau_{x_n}(|v|^{\alpha} v)(s)ds\bigg\|_{L^pL^r}&\lesssim\eps,\\ 
\bigg\|e^{it(\Delta-V)}\int_{-t_n}^{-T}e^{-is(\Delta-V)}\tau_{x_n}(|v|^{\alpha} v)(s)ds \bigg\|_{L^{p}L^{r}}&\lesssim\eps.
\end{aligned}
\end{equation*}
Thus, it remains to prove 
\begin{equation*}
\bigg\|\int_{-T}^{t}e^{i(t-s)\Delta}\tau_{x_n}(|v|^{\alpha} v)(s)ds- \int_{-T}^te^{i(t-s)(\Delta-V)}\tau_{x_n}(|v|^{\alpha} v)(s)ds \bigg\|_{L^{p}L^{r}}\overset{n\ra\infty}\longrightarrow0
\end{equation*}
and we split it by estimating it in the regions mentioned above. 
By using  \eqref{str2.4} and the first one of \eqref{smallness} we get uniformly in $n$ the following estimates:
\begin{equation*}
\begin{aligned}
\bigg\|\int_{-T}^{t}e^{i(t-s)\Delta}\tau_{x_n}(|v|^{\alpha} v)(s)ds\bigg\|_{L^{p}_{(-\infty,-T)}L^{r}}&\lesssim\||v|^{\alpha}v\|_{L^{q'}_{(-\infty,-T)}L^{r'}}\lesssim\varepsilon,\\
\bigg\|\int_{-T}^{t}e^{i(t-s)(\Delta-V)}\tau_{x_n}(|v|^{\alpha} v)(s)ds\bigg\|_{L^{p}_{(-\infty,-T)}L^{r}}&\lesssim\||v|^{\alpha}v\|_{L^{q'}_{(-\infty,-T)}L^{r'}}\lesssim\varepsilon.
\end{aligned}
\end{equation*}
The difference $w_n=\int_{-T}^{t}e^{i(t-s)\Delta}\tau_{x_n}(|v|^{\alpha} v)(s)ds- \int_{-T}^te^{i(t-s)(\Delta-V)}\tau_{x_n}(|v|^{\alpha} v)(s)ds$ satisfies the following Cauchy problem:
\begin{equation*}
\left\{ \begin{aligned}
i\partial_{t}w_n+(\Delta-V)w_n&=-V\int_{-T}^te^{i(t-s)\Delta}\tau_{x_n}(|v|^\a v)(s)\,ds\\
w_n(-T)&=0
\end{aligned}
\right..
\end{equation*}
Then $w_n$ satisfies the integral equation
\begin{equation*}
\begin{aligned}
w_n(t)=\int_{-T}^te^{i(t-s)(\Delta-V)}\bigg(-V\int_{-T}^se^{i(s-\sigma)\Delta}\tau_{x_n}(|v|^\a v)(\sigma)\,d\sigma\bigg)\,ds
\end{aligned}
\end{equation*}
which we estimate in the region $(-T,T)\times\mathbb{R}^d.$ By Sobolev embedding $H^1\hookrightarrow L^r,$ H\"older and Minkowski inequalities we have
therefore  
\begin{equation*}
\begin{aligned}
\bigg\|\int_{-T}^te^{i(t-s)(\Delta-V)}\bigg(-V\int_{-T}^se^{i(s-\sigma)\Delta}\tau_{x_n}(|v|^\a v)(\sigma)\,d\sigma\bigg)\,ds\bigg\|_{L^p_{(-T,T)}L^r}\lesssim\\
\lesssim T^{1/p}\int_{-T}^T\bigg\|(\tau_{-x_n}V)\int_{-T}^se^{i(s-\sigma)\Delta}|v|^\a v(\sigma)\,d\sigma\bigg\|_{H^1}\,ds\lesssim\varepsilon
\end{aligned}
\end{equation*}
by means of Lebesgue's theorem.
\\
What is left is to estimate in the region $(T,\infty)\times\R^d$ the terms 
$$\int_{-T}^te^{i(t-s)(\Delta-V)}\tau_{x_n}(|v|^\a v)\,ds\hbox{\qquad and \qquad}\int_{-T}^te^{i(t-s)\Delta}\tau_{x_n}(|v|^\a v)\,ds.$$
We consider only one term being the same for the other. Let us split the estimate as follows:
\begin{equation*}
\begin{aligned}
\bigg\|\int_{-T}^t&e^{i(t-s)(\Delta-V)}\tau_{x_n}(|v|^\a v)\,ds\bigg\|_{L^p_{(T,\i)}L^r}\leq\\
&\leq\bigg\|\int_{-T}^{T}e^{i(t-s)(\Delta-V)}\tau_{x_n}(|v|^\a v)\,ds\bigg\|_{L^p_{(T,\i)}L^r}\\
&\quad+\bigg\|\int_{T}^te^{i(t-s)(\Delta-V)}\tau_{x_n}(|v|^\a v)\,ds\bigg\|_{L^p_{(T,\i)}L^r}.
\end{aligned}
\end{equation*}
The second term is controlled by Strichartz estimates, and it is $\lesssim\varepsilon$ since we are integrating in the region where $\||v|^\a v\|_{L^{q'}((T,\i);L^{r'})}<\varepsilon$ (by using the second of \eqref{smallness}), while the first term is estimated by using the dispersive relation. More precisely 
\begin{equation*}
\begin{aligned}
\bigg\|\int_{-T}^{T}e^{i(t-s)(\Delta-V)}&\tau_{x_n}(|v|^\a v)\,ds\bigg\|_{L^p_{(T,\i)}L^r}\lesssim\\
&\lesssim\bigg\|\int_{-T}^{T}|t-s|^{-d\left(\frac{1}{2}-\frac{1}{r}\right)}\|v\|^{\alpha+1}_{L^{(\alpha+1)r'}}ds \bigg\|_{L^p_{(T,\infty)}}\\
&\lesssim\bigg\|\int_{\R}|t-s|^{-d\left(\frac{1}{2}-\frac{1}{r}\right)}\|v\|^{\alpha+1}_{L^{(\alpha+1)r'}}
ds \bigg\|_{L^p_{(T,\infty)}}\lesssim\varepsilon
\end{aligned}
\end{equation*}
where in the last step we used Hardy-Sobolev-Littlewood inequality and the fourth of \eqref{smallness}.

\end{proof}

As consequences of the previous lemma we obtain the following corollaries.
\begin{corollary}\label{cor5.6} 
Let $\{x_n\}_{n\in\mathbb{N}}\subset\mathbb{R}^d$ be a sequence such that 
$x_{n,1}
\overset{n\ra \infty}\longrightarrow-\infty$
and let $v\in\mathcal{C}(\mathbb{R};H^1)$ be a solution to 
\eqref{NLS-d} with initial datum $\psi\in H^1$. Then for a sequence $\{t_n\}_{n\in\mathbb{N}}$ such that $|t_n|\overset{n\ra\infty}\longrightarrow\infty$
\begin{equation*}
v_n(t,x)=e^{it(\Delta-V)}\psi_n-i\int_{0}^{t}e^{i(t-s)(\Delta-V)}\big(|v_{n}|^{\alpha}v_{n}\big)ds+e_n(t,x)
\end{equation*}
where $\psi_n:=e^{-it_n(\Delta -V)}\tau_{x_n}\psi,$ $v_n:=v(t-t_n,x-x_n)$ and $\|e_n\|_{L^pL^r}\overset{n\ra\infty}\longrightarrow 0$.  \end{corollary}

\begin{corollary}\label{cor5.7}
Let $\{x_n\}_{n\in\mathbb{N}}\subset\mathbb{R}^d$ be a sequence such that 
$x_{n,1}
\overset{n\ra \infty}\longrightarrow+ \infty$
and let $v\in\mathcal{C}(\mathbb{R};H^1)$ be a solution to 
\eqref{NLS1-d} with initial datum $\psi\in H^1$. Then for a sequence $\{t_n\}_{n\in\mathbb{N}}$ such that $|t_n|\overset{n\ra\infty}\longrightarrow \infty$
\begin{equation*}
v_n(t,x)=e^{it(\Delta-V)}\psi_n-i\int_{0}^{t}e^{i(t-s)(\Delta-V)}\big(|v_{n}|^{\alpha}v_{n}\big)ds+e_n(t,x)
\end{equation*}
where $\psi_n:=e^{-it_n(\Delta -V)}\tau_{x_n}\psi,$ $v_n:=v(t-t_n,x-x_n)$ and $\|e_n\|_{L^pL^r}\overset{n\ra\infty}\longrightarrow 0$.  
\end{corollary}

We shall also need the following results, for whose proof we refer to \cite{BV}.
\begin{prop}\label{prop5.8}
Let $\psi\in H^1.$ There exists $\hat{U}_{\pm}\in\mathcal{C}(\mathbb{R}_{\pm};H^1)\cap L^{p}_{\mathbb{R}_{\pm}}L^r$ solution to \eqref{NLSV-d} such that
\begin{equation*}
\|\hat{U}_{\pm}(t,\cdot)-e^{-it(\Delta-V)}\psi\|_{H^1}\overset{t\rightarrow\pm\infty}\longrightarrow0.
\end{equation*}
Moreover, if $t_n\rightarrow\mp\infty$, then 
\begin{equation*}
\hat{U}_{\pm,n}=e^{it(\Delta-V)}\psi_n-i\int_{0}^{t}e^{i(t-s)(\Delta-V)}\big(|\hat{U}_{\pm,n}|^{\alpha}\hat{U}_{\pm,n}\big)ds+h_{\pm,n}(t,x)
\end{equation*}
where $\psi_n:=e^{-it_n(\Delta-V)}\psi,$ $\hat{U}_{\pm,n}(t,\cdot)=:\hat{U}_{\pm}(t-t_n,\cdot)$ and $\|h_{\pm,n}(t,x)\|_{L^pL^r}\overset{n\ra\infty}\longrightarrow 0.$
\end{prop}

\section{Existence and extinction of the critical element}\label{critical}

In view of the results stated in \autoref{perturbative}, we define 
the following quantity belonging to $(0, \infty]$: 
\begin{align*}
E_{c}=\sup\bigg\{ &E>0 \textit{ such that if } \varphi\in H^1\, \textit{with } E(\varphi)<E\\\notag
&\textit{then the solution of \eqref{NLSV-d} with initial data } \varphi \textit{ is in } L^{p}L^{r}\bigg\}.
\end{align*}
Our aim is to show that $E_c=\infty$ and hence we get the large data scattering.

\subsection{Existence of the Minimal Element}

\begin{prop}\label{lemcri}
Suppose $E_{c}<\infty.$ Then there exists $\varphi_{c}\in H^1$,
$\varphi_{c}\not\equiv0$, such that the corresponding global  solution $u_{c}(t,x)$
to \eqref{NLSV-d} does not scatter. Moreover, there exists $\bar x(t)\in\R^{d-1}$ such that  $\left\{ u_{c}(t, x_1,\bar x-\bar x(t))\right\}_{t\in\R^+} $
is a relatively compact subset in $H^{1}$.
\end{prop}
\begin{proof}
If $E_{c}<\infty$, there exists a sequence $\varphi_{n}$ of elements of $H^{1}$ such that 
\begin{equation*}
E(\varphi_{n})\overset{n\rightarrow\infty}{\longrightarrow}E_{c},
\end{equation*}
and by denoting with $u_{n}\in \mathcal{C}(\mathbb{R};H^{1})$ the corresponding solution to \eqref{NLSV} with initial datum $\varphi_n$ then
\begin{equation*}
u_{n}\notin L^{p}L^{r}.
\end{equation*}
We apply the profile decomposition to $\varphi_{n}:$
\begin{equation}\label{profdec}
\varphi_{n}=\sum_{j=1}^{J}e^{-it_{n}^{j}(-\Delta +V)}\tau_{x_{n}^{j}}\psi^{j}+R_{n}^{J}.
\end{equation}
\begin{claim}\label{claim}
There exists only one non-trivial profile, that is $J=1$.
\end{claim}
Assume $J>1$. For $j\in\{1,\dots, J\}$ to each profile $\psi^{j}$ we associate a nonlinear profile $U_n^j$.  We can have one of the following situations, where we have reordered without loss of generality the cases in these way:
\begin{enumerate}
\item $(t_{n}^{j},x_{n}^{j})=(0,0)\in \R\times \R^d$,
\item $t_{n}^{j}=0$ and  $x_{n,1}^{j}\overset{n \rightarrow \infty}\longrightarrow-\infty,$ 
\item $t_{n}^{j}=0,$ and  $x_{n,1}^{j}\overset{n \rightarrow \infty}\longrightarrow +\infty,$ 
\item $t_{n}^{j}=0,$ $x_{n,1}^{j}=0$ and $|\bar x_{n}^{j}|\overset{n \rightarrow \infty}\longrightarrow\infty,$
\item $x_{n}^{j}=\vec 0$ and $t_{n}^{j}\overset{n \rightarrow \infty}\longrightarrow-\infty$,
\item $x_{n}^{j}=\vec 0$ and $t_{n}^{j}\overset{n \rightarrow \infty}\longrightarrow+\infty$, 
\item $x_{n,1}^{j}\overset{n\to \infty}\longrightarrow-\infty$ and  $t_{n}^{j}\overset{n \rightarrow \infty}\longrightarrow-\infty,$ 
\item $x_{n,1}^{j}\overset{n\to \infty}\longrightarrow-\infty$ and  $t_{n}^{j}\overset{n \rightarrow \infty}\longrightarrow+\infty,$ 
\item $x_{n,1}^{j}\overset{n\to \infty}\longrightarrow+\infty$ and $t_{n}^{j}\overset{n \rightarrow \infty}\longrightarrow-\infty,$ 
\item $x_{n,1}^{j}\overset{n\to \infty}\longrightarrow+\infty$ and $t_{n}^{j}\overset{n \rightarrow \infty}\longrightarrow+\infty,$ 
\item $x_{n,1}^{j}=0,$ $t_{n}^{j}\overset{n \rightarrow \infty}\longrightarrow-\infty$ and $|\bar x^j_n|\overset{n \rightarrow \infty}\longrightarrow\infty,$
\item $x_{n,1}^{j}=0,$ $t_{n}^{j}\overset{n \rightarrow \infty}\longrightarrow+\infty$ and $|\bar x^j_n|\overset{n \rightarrow \infty}\longrightarrow\infty.$
\end{enumerate}

Notice that despite to \cite{BV} we have twelve cases to consider and not six (this is because we have to consider a different behavior of $V(x)$ as $|x|\rightarrow \infty$). 
Since the argument to deal with the cases above is similar 
to the ones considered in \cite{BV} we skip the details. The main point is that
for instance in dealing with the cases 
$(2)$ and $(3)$ above we have to use respectively
\autoref{cor5.3} and \autoref{cor5.4}.
 
When instead $|\bar x_n^j|
\overset{n\to \infty}\longrightarrow \infty$ and $x_{1,n}^j=0$ we use the fact that this sequences can be assumed, according with the profile decomposition 
\autoref{profiledec}
 to have components which are integer multiples of the periods, so the translations and the nonlinear equation commute and if $|t_n|\overset{n\to \infty}\longrightarrow\infty$ we use moreover \autoref{prop5.8}. We skip the details.
Once it is proved that $J=1$ and 
\begin{equation*}
\varphi_n=e^{it_n(\Delta-V)}\psi+R_n
\end{equation*}
with $\psi\in H^1$ and $\underset{n\ra\infty}\limsup\|e^{it(\Delta-V)}R_n\|_{L^pL^r}=0,$ then the existence of the critical element follows now by \cite{FXC}, ensuring that, up to subsequence, $\varphi_n$ converges to $\psi$ in $H^1$ and so $\varphi_c=\psi.$ We define by $u_c$ the solution to \eqref{NLSV-d} with Cauchy datum $\varphi_c,$ and we call it critical element. This is the minimal (with respect to the energy) non-scattering solution to \eqref{NLSV-d}. We can assume therefore with no loss of generality that $\|u_c\|_{L^{p}((0,+\infty);L^r)}=\infty.$ The precompactenss of the trajectory up to translation by a path $\bar x(t)$ follows again by \cite{FXC}.
\end{proof}

\subsection{Extinction of the Minimal Element}

Next we show that the unique solution that satisfies the compactness properties of the minimal element $u_c(t,x)$ (see \autoref{lemcri}) is the trivial solution. Hence we get a contradiction
and we deduce that necessarily $E_c=\infty$.

The tool that we shall use is the following Nakanishi-Morawetz type estimate. 
\begin{lemma}
Let $u(t,x)$ be the solution to \eqref{NLSV-d}, where $V(x)$ satisfies
$x_1 \cdot \partial_{x_1}V (x)\leq 0$ for any $x\in\R^d,$ then
\begin{equation}\label{pote}
\int_\R\int_{\R^{d-1}}\int_\R\frac{t^2|u|^{\a+2}}{(t^2+x_1^2)^{3/2}}\,dx_1\,d\bar x\,dt
<\infty.
\end{equation}
\end{lemma}
\begin{proof}
The proof follows the ideas of \cite{N}; we shall recall it shortly, with the obvious modifications of our context.
Let us introduce
\begin{equation*}
m(u)=a\partial_{x_1}u+gu
\end{equation*}
with 
\begin{equation*}
\begin{aligned}
a=-\frac{2x_1}{\lambda},\quad  g=-\frac{t^2}{\lambda^3}-\frac{it}{\lambda}, \quad 
\lambda=(t^2+x_1^2)^{1/2}
\end{aligned}
\end{equation*} 
and by using the equation solved by $u(t,x)$ we get
\begin{equation}\label{identity}
\begin{aligned}
0&=\Re\{(i\partial_t u+\Delta u-Vu-|u|^\a u)\bar{m})\}
\\
&=\frac{1}{2}\partial_t\bigg(-\frac{2x_1}{\lambda}\Im\{\bar{u}\partial_{x_1}u\}-\frac{t|u|^2}{\lambda}\bigg)\\
& \quad +\partial_{x_1}\Re\{\partial_{x_1}u\bar{m}-al_V(u)-\partial_{x_1}g\frac{|u|^2}{2}\}\\
& \quad +\frac{t^2G(u)}{\lambda^3}+\frac{|u|^2}{2}\Re\{\partial_{x_1}^2g\}\\
& \quad +\frac{|2it\partial_{x_1}u+{x_1}u|^2}{2\lambda^3}-{x_1}\partial_{x_1}V\frac{|u|^2}{\lambda}\\
& \quad +div_{\bar x}\Re\{\bar m\nabla_{\bar x}u\}.
\end{aligned}
\end{equation} 
with $G(u)=\frac{\a}{\a+2}|u|^{\a+2},$ $l_V(u)=\frac{1}{2}\left(-\Re\{i\bar{u}\partial_tu\}+|\partial_{x_1}u|^2+\frac{2|u|^{\a+2}}{\a+2}+V|u|^2\right)$ and $div_{\bar x}$ is the divergence operator w.r.t. the $(x_2,\dots,x_d)$ variables. 
Making use of the repulsivity assumption in the $x_1$ direction, we get \eqref{pote} by integrating \eqref{identity} on $\{1<|t|<T\}\times\R^d,$ obtaining 
\begin{equation*}
\int_1^T\int_{\R^{d-1}}\int_\R\frac{t^2|u|^{\a+2}}{(t^2+x_1^2)^{3/2}}\,dx_1\,d\bar x\,dt\leq C,
\end{equation*}
where $C=C(M,E)$ depends on mass and energy and then letting $T\to\infty.$
\end{proof}
\begin{lemma}\label{limit-point}
Let $u(t,x)$ be a nontrivial solution to \eqref{NLSV-d} such that  
for a suitable choice $\bar x(t)\in \R^{d-1}$ we have that
$\{u(t,x_1, \bar x-\bar x(t))\}\subset H^1$ is a precompact set.
If $\bar{u}\in H^1$ is one of its limit points, then $\bar{u}\neq0.$
\end{lemma}
\begin{proof}
This property simply follows from the conservation of the energy.
\end{proof}

\begin{lemma}\label{lem2}
If $u(t,x)$ is an in \autoref{limit-point} then for any $\varepsilon>0$ there exists $R>0$ such that 
\begin{equation}
\sup_{t\in \R} \int_{\R^{d-1}} \int_{|x_1|>R} (|u|^2+|\nabla_x u|^2+|u|^{\alpha+2})\,d\bar x\,dx_1<\varepsilon.
\end{equation} 
\end{lemma}
\begin{proof}
This is a well-known property implied by the precompactness of the sequence. 
\end{proof}
\begin{lemma}\label{lem1}
If $u(t,x)$ is an in \autoref{limit-point}  
then there exist $R_0>0$ and $\varepsilon_0>0$ such that 
\begin{equation}
\int_{\R^{d-1}}\int_{|x_1|<R_0}|u(t,x_1,\bar x-\bar x(t))|^{\alpha+2}\,d\bar x\,dx_1>\varepsilon_0 \qquad \forall\,t\in\R^+.
\end{equation} 
\end{lemma}
\begin{proof}
It is sufficient to prove that $\inf_{t\in \R^+} \|u(t ,x_1,\bar x-\bar x(t))\|_{L^{\alpha+2}}
>0$, then the result follows by combining this fact with
\autoref{lem2}.
If by the absurd it is not true then 
there exists a sequence $\{t_n\}_{n\in \N}\subset\R^+$ such that 
$u(t_n ,x_1,\bar x-\bar x(t_n))
\overset{n\ra\infty}\longrightarrow 0$ in $L^{\alpha+2}.$
On the other hand by the compactness assumption, it implies that
$u(t_n ,x_1,\bar x-\bar x(t_n))
\overset{n\ra\infty}\longrightarrow 0$ in $H^{1}$,
and it is in contradiction with \autoref{limit-point}.
\end{proof}

We now conclude the proof of scattering for large data, by showing the extinction
of the minimal element. 
Let $R_0>0$ and $\varepsilon_0>0$ be given by \autoref{lem1}, then
\begin{equation*}
\begin{aligned}
\int_\R\int_{\R^{d-1}}\int_\R\frac{|u|^{\alpha+2}t^2}{(t^2+x_1^2)^{3/2}}\,dx_1\,d\bar x\,dt&\geq\int_\R \int_{\R^{d-1}}\int_{|x_1|<R_0}\frac{t^2|u(t,x_1,\bar x-\bar x(t))|^{\alpha+2}}{(t^2+x_1^2)^{3/2}}\,dx_1\,d\bar x\,dt\\
&\geq\varepsilon\int_{1}^T\frac{t^2}{(t^2+R_0^2)^{3/2}}\,dt\to\infty\qquad \text{if}\quad T\to\infty.
\end{aligned}
\end{equation*}
Hence we contradict \eqref{pote} and we get that the critical element
cannot exist.
 
\section{Double scattering channels in $1D$}

This last section is devoted to prove \autoref{linscat}. Following \cite{DaSi} (see \emph{Example 1,} page $283$) we have the following property:
\begin{equation}\label{multi}
\begin{aligned}
\forall\, \psi\in L^2 \quad &\exists\,\eta_\pm, \gamma_\pm\in L^2 \hbox{ \,such that } \\
\|e^{it (\partial_x^2 - V)} \psi &-e^{it\partial_x^2}\eta_\pm-e^{it(\partial_x^2 -1)}\gamma_\pm\|_{L^2}\overset{t \rightarrow\pm \infty}\longrightarrow 0.
\end{aligned}
\end{equation}

Our aim is now to show that \eqref{multi} actually holds in $H^1$ provided that $\psi\in H^1$. We shall prove this property for $t\rightarrow +\infty$ (the case $t\rightarrow -\infty$
is similar).

\subsection{Convergence \eqref{multi} occurs in $H^1$ provided that $\psi\in H^2$}

In order to do that it is sufficient to show that 
\begin{equation}\label{firststep}\psi\in H^2 \Longrightarrow
\eta_+, \gamma_+\in H^2.\end{equation}
Once it is proved then we conclude the proof of this first step by using the following interpolation
inequality 
$$\|f\|_{H^1}\leq \|f\|_{L^2}^{1/2} \|f\|_{H^2}^{1/2}$$
in conjunction with \eqref{multi} and with the bound 
$$ 
\sup_{t\in \mathbb R} \|e^{it (\partial_x^2 - V)} \psi -e^{it\partial_x^2}\eta_+-e^{it(\partial_x^2 -1)}\gamma_+\|_{H^2}<\infty$$
(in fact this last property follows by the fact that $D(\partial_x^2 - V(x))=H^2$ is preserved along the linear flow and by \eqref{firststep}).
Thus we show \eqref{firststep}. 
Notice that by \eqref{multi} we get
\begin{equation*}
\|e^{-it\partial_x^2}e^{it (\partial_x^2 - V)} \psi -\eta_+-e^{-it}\gamma_+\|_{L^2}\overset{t \rightarrow \infty}\longrightarrow 0,
\end{equation*}
and by choosing as subsequence $t_n=2\pi n$ we get
\begin{equation*}
\|e^{-it_n\partial_x^2}e^{it_n (\partial_x^2 - V)} \psi -\eta_+-\gamma_+\|_{L^2}\overset{n \rightarrow \infty}\longrightarrow 0.
\end{equation*}
By combining this fact with the bound 
$\sup_n \|e^{-it_n\partial_x^2}e^{it_n (\partial_x^2 - V)} \psi\|_{H^2}<\infty$
we get 
$\eta_++\gamma_+\in H^2$. 
Arguing as above but by choosing $t_n=(2n+1)\pi$ we also get
$\eta_+-\gamma_+\in H^2$ and hence necessarily $\eta_+, \gamma_+\in H^2$.

\subsection{The map $H^2\ni \psi\mapsto (\eta_+, \gamma_+)\in H^2\times H^2$ satisfies $\|\gamma_+\|_{H^1}+\|\eta_+\|_{H^1}\lesssim \|\psi\|_{H^1}$}
Once this step is proved then we conclude by a straightforward density argument.
By a linear version of the conservation laws \eqref{consmass}, \eqref{consen} 
we get 
\begin{equation}\label{H1V}
\|e^{it (\partial_x^2 - V)} \psi\|_{H^1_V}=
\|\psi\|_{H^1_V}
\end{equation}
where
$$
\|w\|_{H^1_V}^2
=\int |\partial_x w|^2 dx+\int V|w|^2dx+\int |w|^2 dx.
$$
Notice that this norm is clearly equivalent to the usual norm of $H^1$.
\\
Next notice that by using the conservation of the mass 
we get
\begin{equation*}
\|\eta_++\gamma_+\|_{L^2}^2=\|\eta_+ + e^{-2n\pi i}\gamma_+\|_{L^2}^2
=\|e^{i2\pi n\partial_x^2}\eta_+ + e^{i2\pi n(\partial_x^2-1)}\gamma_+\|_{L^2}^2
\end{equation*}
and by using 
\eqref{multi} we get
\begin{equation*}\|\eta_++\gamma_+\|_{L^2}^2=\lim_{t\rightarrow\infty}\|e^{it( \partial_x^2 - V)} \psi\|_{L^2}^2
=\|\psi\|_{L^2}^2
\end{equation*}
Moreover we have
\begin{align}\notag
\|\partial_x(\eta_++\gamma_+)\|^2_{L^2}&=\|\partial_x(\eta_++e^{-2n\pi i}\gamma_+)\|^2_{L^2}=\|\partial_x(e^{i2\pi n\partial_x^2}(\eta_++e^{-i2\pi n}\gamma_+))\|^2_{L^2}\\\notag
&=\|\partial_x(e^{i2\pi n \partial_x^2}\eta_++e^{i2\pi n (\partial_x^2-1)}\gamma_+)\|^2_{L^2}\\\notag
\end{align}
and by using the previous step and \eqref{H1V} we get
\begin{align*}
\|\partial_x(\eta_++\gamma_+)\|^2_{L^2}=&\lim_{t\rightarrow+\infty}\|\partial_x(e^{it(\partial_x^2 - V)}\psi)\|^2_{L^2}\\
\leq&\lim_{t\rightarrow \infty}\|e^{it(\partial_x^2 - V)}\psi \|^2_{H^1_V}=\|\psi\|^2_{H^1_V}
\lesssim \|\psi\|^2_{H^1}.
\end{align*}
Summarizing we get 
$$\|\eta_++\gamma_+\|_{H^1}\lesssim \|\psi\|_{H^1}.$$
By a similar argument and by replacing the sequence $t_n=2\pi n$ by 
$t_n=(2n+1)\pi$ we get
$$\|\eta_+-\gamma_+\|_{H^1}\lesssim \|\psi\|_{H^1}.$$
The conclusion follows.

\section*{acknowledgements}

The first author would like to thank Professor Kenji Yajima for having brought the theory of double channel scattering to his knowledge.
The second author is supported by the research project PRA 2015. 
Both authors are grateful to the referee for useful remarks and comments that
improved a previous version of the paper.


\end{document}